\documentclass[11pt]{amsart}

\usepackage{amssymb,amsfonts,latexsym}
\usepackage[centertags]{amsmath}
\usepackage{amsfonts}
\usepackage{amssymb}
\usepackage{amsthm}
\usepackage[all]{xy}
\usepackage{cite}
\usepackage{graphicx,color}
\usepackage{todo}

\newtheorem{cont}{cont}[section]
\newtheorem{theorem}[cont]{Theorem}
\newtheorem{proposition}[cont]{Proposition}
\newtheorem{lemma}[cont]{Lemma}
\newtheorem{corollary}[cont]{Corollary}
\newtheorem{definition}[cont]{Definition}

\newtheorem{construction}[cont]{Construction}

\newtheorem*{Enunciato*}{Enunciato}
\newtheorem{conj}[cont]{Conjecture}
\newtheorem{notation}[cont]{Notation}
\numberwithin{equation}{section}
\newtheorem{remark}[cont]{Remark}

\newtheorem*{not*}{Notation}

\newcommand{\cK}{{\mathcal K}}
\newcommand{\cO}{{\mathcal O}}
\newcommand{\shF}{\mathcal{F}}

\newcommand{\shH}{\mathcal{H}}
\newcommand{\cA}{\mathcal{A}}
\newcommand{\cB}{\mathcal{B}}
\newcommand{\cC}{\mathcal{C}}
\newcommand{\cD}{\mathcal{D}}

\newcommand{\cL}{\mathcal{L}}
\newcommand{\shM}{\mathcal{M}}
\newcommand{\shE}{\mathcal{E}}
\newcommand{\PP}{\mathbb{P}}
\newcommand{\ZZ}{\mathbb{Z}}

\newcommand{\odi}[1]{\mathcal{O}_{#1}}

\newcommand{\arr}{\longrightarrow}

\DeclareMathOperator{\Hl}{H} \DeclareMathOperator{\h}{h}
 
\DeclareMathOperator{\rk}{rk} \DeclareMathOperator{\Hom}{Hom}
 
\DeclareMathOperator{\de}{deg}

\DeclareMathOperator{\depth}{depth}

\DeclareMathOperator{\Proj}{Proj} \DeclareMathOperator{\di}{dim}
\DeclareMathOperator{\codim}{codim}
\DeclareMathOperator{\coker}{coker}
\DeclareMathOperator{\Ext}{Ext} \DeclareMathOperator{\pd}{pd}
 \DeclareMathOperator{\ext}{ext}
 \DeclareMathOperator{\im}{im}

\begin{document}
\title{The representation type of  determinantal varieties
 }

\author[Jan O.\ Kleppe, Rosa M.\ Mir\'o-Roig]{Jan O.\ kleppe, Rosa M.\
Mir\'o-Roig$^{*}$}
\address{Oslo and Akershus University College,
         Faculty of Technology, Art and Design,
         PB. 4, St. Olavs Plass, N-0130 Oslo,
         Norway}
\email{JanOddvar.Kleppe@hioa.no}
\address{Facultat de Matem\`atiques i Inform\`{a}tica,
Universitat de Barcelona, Gran Via de les Corts Catalanes
585, 08007 Barcelona, SPAIN } \email{miro@ub.edu}

\date{\today} \thanks{$^*$ Partially supported by MTM2013-45075-P}

\begin{abstract} This work is entirely devoted to construct huge families of indecomposable arithmetically Cohen-Macaulay (resp. Ulrich) sheaves $\shE$ of arbitrary high rank  on a general standard (resp. linear) determinantal scheme $X\subset \PP^n$ of codimension $c\ge 1$, $n-c\ge 1$ and defined by the maximal minors of a $t\times (t+c-1)$
homogeneous matrix $\cA$. The sheaves $\shE$ are constructed as iterated extensions of sheaves of lower rank. As applications: (1) we prove that any general standard determinantal scheme $X\subset \PP^n$ is of wild representation type provided the degrees of the entries of the matrix $\cA$ satisfy some  weak numerical assumptions; and (2) we determine values of $t$, $n$ and $n-c$ for which a linear standard determinantal scheme $X\subset \PP^n$ is of wild representation type  with respect to  the much more restrictive category of its indecomposable Ulrich sheaves, i.e. $X$ is of Ulrich wild representation type. \end{abstract}

\maketitle

\tableofcontents

\section{Introduction}

\vskip 2mm
The goal of this paper is to construct huge families of indecomposable maximal Cohen-Macaulay $A$-modules $E$ where  $A$ is  the homogeneous coordinate ring of a standard determinantal scheme, i.e. $A$ is the quotient ring of codimension $c$ of a polynomial ring $K[x_0, \cdots ,x_n]$ by the homogeneous ideal generated  by the  maximal minors of a  homogeneous $t\times (t+c-1)$ matrix.
  Our interest in this topic is based on the geometric counterpart. Given a projective variety $X\subset \PP^n$ we would like to understand the complexity of $X$ in terms of the associated category of arithmetically Cohen-Macaulay sheaves that it supports and, in particular, we are interested in
the existence of  projective varieties supporting  families of  arbitrarily high rank and dimension of indecomposable arithmetically Cohen-Macaulay
 bundles,  i.e. vector bundles $\shE$ without intermediate cohomology (ACM bundles, for short). In the case of linear standard determinantal schemes (i.e. schemes defined by the maximal minors of a matrix with linear entries) the ACM bundles  $\shE$ that we will construct will share another important feature, namely, the associated maximal Cohen-Macaulay module $E=\Hl_*^0(\shE)$ will be maximally generated. In \cite{Ulr}, Ulrich proved that  the maximal number of generators $m(E)$ of a maximal Cohen-Macaulay module $E$  associated to an ACM sheaf $\shE$ on a projective ACM variety $X$ is bounded by $m(E)\le \deg(X)·\rk(\shE)$. ACM sheaves attaining this bound are called Ulrich sheaves.
   The existence of Ulrich sheaves on a projective variety  is a challenging problem which has lately received a lot of attention. The   interest on Ulrich sheaves  relies among other things
on the fact that a $d$-dimensional scheme $X\subset \PP^n$ supports an Ulrich sheaf (resp. bundle) if and only
if the cone of cohomology tables of coherent sheaves  (resp. bundles) on $X$ coincides with
the cone of cohomology tables of coherent sheaves (resp. bundles) on $\PP^d$ (cf. \cite{ES}).  From the algebraic point of view the existence of an Ulrich $R$-module sheds light over the structure of $R$. In fact, we have the following criterion of Gorensteiness: if a Cohen-Macaulay ring $R$ supports an Ulrich $R$-module $M$ then $R$ is Gorenstein if and only if $\Ext^{i}_R(M,R)=0$ for $1\le i \le \dim R$ (cf. \cite{Ulr}).

Since the seminal
result by Horrocks characterizing ACM bundles on $\PP^n$ as those that
completely split into a sum of line bundles (cf. \cite{Hor}), the study of the category of indecomposable ACM bundles   on a
projective variety X becomes a natural way to measure the
complexity of the underlying variety X. Mimicking an analogous trichotomy in Representation
Theory, it has been proposed a classification of ACM projective varieties  as finite, tame or wild
(see Definition \ref{FineTameWild}) according to the complexity of their  category of ACM
bundles. This trichotomy is exhaustive for the case of ACM curves: rational
curves are finite (see \cite{BGS}; Theorem C and \cite{EH}; pg. 348), elliptic curves are tame (see \cite{A}; Theorems 7 and 10) and curves of higher genus are wild (see \cite{DG}; Theorem 16); and one of the main achievements in this
field has been the classification of  varieties of finite representation type in a very short list (cf. \cite{BGS}; Theorem C and \cite{EH}; pg. 348). However, it remains open to find out the representation type of the remaining ones.

The first goal of this paper is to construct families of indecomposable ACM bundles of arbitrary high rank and dimension on standard determinantal schemes $X$ and
to conclude that $X$ is of wild representation type provided the degree matrix associated to $X$ satisfies some  weak numerical hypothesis.
As we pointed out before, among ACM vector bundles $\shE$ on a  variety $X$, it is interesting to spot a very important
subclass for which its associated module $\oplus _{t} \Hl^0(X, \shE(t))$ has the maximal number of generators,
which turns out to be $\deg(X){\cdot} \rk(\shE)$. It is therefore
a meaningful question to find out if a given projective variety X is of wild representation
type with respect to the much more restrictive category of its indecomposable Ulrich
bundles.
In the second part of this  paper, we are going to focus our attention on linear standard determinantal schemes $X\subset \PP^n$ associated to a $t\times (t+c-1)$ matrix with entries linear forms and we are going to determine values of $t$, $n$ and $c$ for which $X$ is of wild representation
type with respect to the much more restrictive category of its indecomposable Ulrich sheaves (i.e. $X$ has Ulrich wild representation type). As  classical examples of linear standard determinantal schemes we have rational normal scrolls, Segre varieties, etc.

\vskip 2mm Let us briefly explain how this paper is organized. In Section 2, we fix notation and we collect  the background and basic results on (linear) standard determinantal schemes and the associated complexes needed in the sequel. Section 3 contains one of the main results of this work (Theorem \ref{Bigthm}). In this theorem we prove that a general determinantal scheme $X\subset \PP^n$ of codimension $c\ge 1$ and dimension $n-c\ge 2$ is of wild representation type provided the degree  of the entries of its associated matrix verify some weak numerical assumptions. To achieve our result we construct huge families of indecomposable ACM sheaves on $X$ of arbitrarily high rank as iterated extensions of ACM sheaves of lower rank. The existence of  such extensions is ensured by a series of technical lemmas gathered in the beginning of the section.

In sections 4 and 5 we deal with standard determinantal schemes $X\subset \PP^n$ of codimension $c$ defined by the maximal minors of a $t\times (t+c-1)$ matrix with linear entries and we address the problem of determining the values of $c$, $n-c$ and $t$ for which    wild representation type is witnessed  by means of Ulrich bundles; i.e. values of $c$, $n-c$ and $t$ for which $X$ has Ulrich wild representation type.
The approach is analogous to the one developed in section 3; starting with rank 1 Ulrich sheaves instead of rank 1 ACM sheaves we are able to construct huge families of indecomposable Ulrich sheaves of high rank and to conclude that under some numerical assumptions on $c$, $n-c$ and $t$, $X$ is of Ulrich wild representation type (see  Theorem \ref{XdL21}, Remark  \ref{ext1(L2,L1)t=2,3}   and Theorems \ref{mainthm11} and \ref{mainthm2}).  We end the paper with a Conjecture raised by this paper and proved in many cases (cf.
Conjecture \ref{conjecture}).

We have become aware of the preprint \cite{FPL} where the
representation type of integral ACM varieties is considered and it is
  shown that all ACM projective integral varieties which are not cones  are  of  wild  CM  type,  except  for  a  few  cases  which  they
completely classify. Our works are independent and some of our proofs can be shortened (for instance Theorems \ref{Bigthm} and \ref{mainthm11}). We decided to keep at least a short sketch of all of them because for the case of standard determinantal varieties the results that we get are slightly stronger. Indeed, under some weak numerical assumptions, we prove that generic standard determinantal   (resp. linear standard determinantal) varieties are not only wild but also strictly wild. Moreover, our construction is explicit and it allows us to control the rank of the ACM (resp. Ulrich) sheaves that we construct (they usually have lower rank than the rank of the sheaves constructed in \cite{FPL}) and we can explicitly compute a formula $f(r)$ for the dimension of the families of simple ACM (resp. Ulrich) sheaves of rank $r$ that we build.

 {\bf Acknowledgement.} The first author would like to thank the
University of Barcelona for its hospitality during his visit to
Barcelona in March 2015 where one of our final main theorems (Theorem
3.14) was established. The authors thank the referee for his/her useful comments.

\vskip 2mm
\noindent \underline{Notation.} Throughout this paper $K$ will be an algebraically closed field, $R=K[x_0, x_1, \cdots ,x_n]$, $\mathfrak{m}=(x_0, \ldots,x_n)$ and $\PP^n=\Proj(R)$. Given a closed subscheme $X\subset \PP^n$, we denote by  ${\mathcal I}_X$ its ideal sheaf and $ I(X)=H^0_{*}(\PP^n, {\mathcal I}_X)$
its saturated homogeneous ideal unless $X=\emptyset $, in which case we let $I(X)=\mathfrak{m}$. If $X$ is equidimensional and
Cohen-Macaulay of codimension $c$, we set $\cK_X ={\mathcal
E}xt^c_{{\mathcal O}_{\PP^n}} ({\mathcal O}_X,{\mathcal
O}_{\PP^n})(-n-1)$ to be its canonical sheaf. Given  a coherent sheaf $\shE$ on $X$ we are going to denote the twisted sheaf $\shE\otimes\odi{X}(l)$ by $\shE(l)$. As usual, $\Hl^i(X,\shE)$ stands for the cohomology groups, $\h^i(X,\shE)$ for their dimension and $\Hl^i_*(X,\shE)=\oplus _{l \in \ZZ}\Hl^i(X,\shE(l))$. We also set $\ext^i(\shE,\shF):=\di_K\Ext^i(\shE,\shF)$.

For any graded Cohen-Macaulay quotient $A$ of $R$ of codimension $c$,
we let $I_A=\ker(R\twoheadrightarrow A)$ and  $K_A=\Ext^c_R (A,R)(-n-1)$ be its canonical module. When we
write $X=\Proj(A)$, we let $A=R/I(X)$ and $K_X=K_A$. If $M$ is a
finitely generated graded $A$-module, let $\depth_{J}{M}$ denote
the length of a maximal $M$-sequence in a homogeneous ideal $J$
and let $\depth {M} = \depth_{\mathfrak m}{ M}$.


\section{Background and preparatory results}
For convenience of the reader we include in this section the
background and basic results on  standard determinantal varieties as well as on ACM and Ulrich sheaves needed later on.
Let us start gathering together the results on
standard
determinantal schemes and the associated complexes needed  in the sequel and we refer
to \cite{b-v}, \cite{eise} and \cite{rm} for more details.
\begin{definition}
\rm If  $\cA$ is a homogeneous matrix, we denote by  $I(\cA)$ the
ideal of $R$ generated by the  maximal minors of $\cA$ and by $I_j(\cA)$
the ideal generated by the $j \times j$ minors of $\cA$. A codimension $c$
subscheme
$X\subset
\PP^n$ is called a
\emph{standard determinantal} scheme if $I(X)=I(\cA)$ for some
$t\times (t+c-1)$ homogeneous matrix $\cA$. In addition, we will say that  $X$  is  a \emph{linear standard  determinantal scheme} if all entries of $\cA$ are linear forms.
 \end{definition}

Denote by
$\varphi
:F\longrightarrow G$ the morphism of free graded $R$-modules of
rank $t+c-1$ and $t$, defined by the homogeneous matrix $\cA$ of $X$, and
 by ${\cC}_i(\varphi )$ the (generalized) Koszul complex:
$${\cC}_i(\varphi): \  0 \rightarrow \wedge^{i}F \otimes S_{0}(G)\rightarrow
\wedge^{i-1} F
\otimes S
_{1}( G)\rightarrow \ldots \rightarrow \wedge ^{0} F \otimes S_i (G) \rightarrow 0.$$

 Let ${\cC}_i(\varphi)^*$ be the $R$-dual of
${\cC}_i(\varphi)$. The dual map $\varphi^*$ induces graded
morphisms $$\mu_i:\wedge ^{t+i}F\otimes
\wedge^tG^*\rightarrow \wedge^{i}F.$$

They can be used to splice the complexes ${\cC}_{c-i-1}(\varphi)^*\otimes \wedge^{t+c-1}F\otimes \wedge^tG^*$ and
 ${\cC}_i(\varphi)$ to a complex ${\cD}_i(\varphi):$
 \begin{equation}\label{splice}
0 \rightarrow \wedge^{t+c-1}F \otimes S_{c-i-1}(G)^*\otimes
\wedge^tG^*\rightarrow \wedge^{t+c-2} F \otimes S _{c-i-2}(G)^*\otimes
\wedge ^tG^*\rightarrow \ldots \rightarrow \end{equation}
$$\wedge^{t+i}F \otimes
S_{0}(G)^*\otimes \wedge^tG^*\stackrel{\mu _{i}}{\longrightarrow} \wedge^{i} F \otimes S _{0}(G)
\rightarrow \wedge ^{i-1} F \otimes S_1(G)\rightarrow \ldots \rightarrow \wedge^0F\otimes
S_i(G)\rightarrow 0 .$$

The complex  ${\cD}_0(\varphi)$ is called the {\em Eagon-Northcott
complex} and the complex  ${\cD}_1(\varphi)$ is called the
{\em Buchsbaum-Rim complex}. Let us rename the complex  ${\cC}_c(\varphi)$ as  ${\cD}_c(\varphi)$. Denote by $I_m(\varphi)$ the
ideal generated by the $m\times m$ minors of the matrix $\cA$
representing $\varphi$. Letting $S_{-1}M:=\Hom(M,R/I(\cA))$ we have the following well known result:

\begin{proposition}\label{resol} Let $X \subset \PP^n$ be a standard determinantal subscheme
of codimension $c$ associated to a graded minimal (i.e.
$\im (\varphi)\subset \mathfrak{m}G$) morphism $\varphi: F\rightarrow G$ of free
$R$-modules of rank $t+c-1$ and $t$, respectively. Set $M= \coker(
\varphi )$. Then it holds:

(i) ${\cD}_i(\varphi)$ is acyclic for $-1\le i \le c$.

(ii) ${\cD}_0(\varphi)$ is a minimal free graded $R$-resolution of
$R/I(X)$  and  ${\cD}_i(\varphi)$ is a minimal free graded
$R$-resolution of length $c$ of $S_iM$, $-1\le i \le c$ .

(iii) $K_X \cong S_{c-1}M(\mu )$ for some $\mu \in \ZZ$. So, up to twist,
$\mathcal{D}_{c-1}(\varphi)$ is a minimal free graded $R$-module resolution of
$K_X$.

(iv) ${\cD}_i(\varphi)$ is a minimal free graded $R$-resolution of $S_iM$
for $c+1\le i$ whenever $\depth_{I_m(\varphi)}R\ge t+c-m$ for every $m$ such
that $t\ge m \ge max(1,t+c-i)$.
\end{proposition}

\begin{proof}
See, for instance \cite{b-v}; Theorem 2.20  and \cite{eise}; Theorem A2.10 and Corollaries A2.12 and
A2.13.
\end{proof}

Let us also recall
 the following useful comparison of cohomology
groups. If $Z\subset X$ is a closed subset such that $U=X\setminus Z$ is a local complete intersection,  $L$ and $N$ are finitely generated $R/I(X)$-modules, $\widetilde{N}$ is locally free on
$U$ and
 $\depth_{I(Z)}L\ge r+1$, then the natural map
\begin{equation} \label{NM}
\Ext^{i}_{R/I(X)}(N,L)\longrightarrow
\Hl_{*}^{i}(U,{\mathcal H}om_{\odi{X}}(\widetilde{N},\widetilde{L}))\cong \oplus _{\nu \in \ZZ}\Ext^{i}_{\odi{X}}(\widetilde{N},\widetilde{L}(\nu ))
\end{equation}
is a degree-preserving isomorphism, (resp. an injection) for $i<r$ (resp. $i=r$)
cf. \cite{SGA2}; expos\'{e} VI. Recall that we interpret $I(Z)$ as $\mathfrak{m}$ if $Z=\emptyset $.

\vskip 2mm
We end this section  setting some preliminary notions  concerning  the definitions and basic results on  ACM sheaves as well as on Ulrich sheaves.

\begin{definition}\label{ACM} \rm
Let $X\subset \PP^n$ be a projective scheme and let $\shE$ be a coherent sheaf on $X$. We say that $\shE$  is  \emph{arithmetically Cohen Macaulay} (shortly, ACM) if it is locally Cohen-Macaulay (i.e., $\depth \shE_x=\di \odi{X,x}$ for every point $x\in X$) and has no intermediate cohomology, i.e.
$$
\Hl^i(X,\shE (t))=0 \quad\quad \text{    for all $t$ and $i=1, \ldots , \di X-1.$}
$$
\end{definition}
Notice that when $X$ is a non-singular variety, any coherent ACM sheaf on $X$ is locally free.  ACM sheaves are closely related to their algebraic counterpart, the maximal Cohen-Macaulay modules.

\begin{definition} \rm
A graded $A$-module $E$ is a \emph{maximal Cohen-Macaulay} module (MCM for short) if $\depth E=\di E=\di A$.
\end{definition}

In fact, on any ACM scheme $X\subseteq\PP^n$, there exists a bijection between ACM sheaves $\shE$ on $X$ and MCM $A$-modules $E$ given by
the functors $E\rightarrow \widetilde{E}$ and $\shE\rightarrow \Hl^0_*(X,\shE)$.

\begin{definition}\rm
Given a closed subscheme $X\subset \PP^n$ of dimension $d>0$, a torsion free sheaf $\shE$ on  $X$
  is said to be \emph{initialized} if
$$
\Hl^0(X,\shE(-1))=0 \ \ \text{ but } \  \Hl^0(X,\shE)\neq 0.
$$
 If $\shE$ is a locally Cohen-Macaulay sheaf then, there  exists a unique integer $k$ such that $\shE_{init}:=\shE(k)$ is
initialized.
\end{definition}

\begin{definition} \rm
Given a projective scheme $X\subset \PP^n$ and a coherent sheaf $\shE$ on $X$, we say that $\shE$ is an \emph{Ulrich sheaf} if  $\shE$ is an ACM sheaf and $\h^0(\shE_{init})=\deg(X)\rk(\shE)$.
\end{definition}

The following result justifies this definition:
\begin{theorem}
Let $X\subseteq\PP^n$ be an integral subscheme and let $\shE$  be an ACM sheaf on $X$. Then the minimal number of generators $m(\shE)$ of the $A$-module $\Hl^0_*(\shE)$ is bounded by
$$
m(\shE)\leq \deg (X)\rk(\shE).
$$
\end{theorem}

Therefore, since it is obvious that for an initialized sheaf $\shE$, $\h^0(\shE)\leq m(\shE)$, the minimal number of generators of Ulrich sheaves is
as large as possible. Modules attaining this upper bound were studied by Ulrich in \cite{Ulr} and ever since modules  with this property are called Ulrich modules.
In \cite{ESW}; pg. 543, Eisenbud, Schreyer and Weyman left open the problem: Is every variety $X\subset \PP^n$ the support of an Ulrich sheaf? If so, what is the smallest possible rank for such a sheaf?  The existence of Ulrich sheaves on a projective variety is a challenging problem since few examples are known. For recent results on Ulrich sheaves the reader can see \cite{CH2}, \cite{CMP}, \cite{CM}, \cite{FM}, \cite{MR2013}, \cite{MRP}, \cite{MRP2} and the references quoted there.
This paper is entirely
devoted to construct new families of indecomposable ACM sheaves on standard determinantal varieties and Ulrich sheaves on linear standard determinantal varieties.


\section{The representation type of a determinantal variety}

A possible way to classify ACM varieties is according to the complexity of the category of ACM sheaves that they support.
Recently, inspired by an analogous classification for quivers and for $K$-algebras of finite type, it has been proposed the classification of any  ACM variety as being of \emph{finite, tame or wild representation type} (cf. \cite{DG} for the case of curves and \cite{CH} for the higher dimensional case). Let us introduce these definitions slightly modifying  the usual ones (see also \cite{MR2013} and \cite{MR2014}):

\begin{definition} \label{FineTameWild} \rm  Let $X\subseteq\PP^n$ be an ACM scheme of dimension $d$.

 (i) We say that $X$ is of {\em finite representation type} if it has, up to twist and isomorphism, only a finite number of indecomposable ACM sheaves.

(ii)  $X$ is of {\em tame representation type} if either it has, up to twist and isomorphism, an infinite
 discrete set of indecomposable ACM sheaves or,
 for each rank $r$, the indecomposable ACM sheaves of rank $r$ form a finite number of families of dimension at most $1$.

 (iii) $X$ is of {\em wild representation type} if there exist $l$-dimensional families of non-isomorphic indecomposable ACM sheaves for arbitrary large $l$.

 (iv) $X$ is of {\em Ulrich wild representation type} if there exist $l$-dimensional families of non-isomorphic indecomposable Ulrich sheaves for arbitrary large $l$.
 \end{definition}

 For the sake of completeness we would like to comment the above definitions and other definitions of wildness that we can find in the literature.

\begin{remark} \rm
 (1) Usually tameness is defined  without allowing
 an infinite
 discrete set of indecomposable ACM sheaves. We enlarge the category of varieties of tame representation type because we really want the trichotomy (i) - (iii) to be exhaustive for any $d$-dimensional projective ACM variety.
 This trichotomy is exhaustive for the case of ACM curves: rational
curves are finite, elliptic curves are tame and curves of higher genus are wild.
If we use the standard definition of tameness, then the trichotomy will not be exhaustive
because  in \cite{CaH} it was  proved
 the quadratic cone in $\PP^3$ has an infinite
 discrete set of indecomposable ACM sheaves.

(2) The notion of Ulrich wild representation type was introduced in \cite{FPL} although  in \cite{CMP} the authors already addressed the problem of finding out if a given projective variety X is of wild representation type with respect to the much more restrictive category of its indecomposable Ulrich vector bundles.

(3) In \cite[Definition 1.4]{DG}, Drozd and Greuel  introduced two definitions of wildness: geometrically wild and algebraically wild. Roughly speaking a projective variety $X\subset \PP^n$ is {\em geometrically wild} if the corresponding homogeneous coordinate ring $A$ has arbitrarily large families of indecomposable maximal Cohen-Macaulay  $A$-modules. $X$ is said to be {\em algebraically wild} if
for every finitely generated $K$-algebra $ \Lambda $ there exists a family $\shM$
of maximal Cohen-Macaulay $A$-modules over $\Lambda $ (i.e. a finitely generated $(A,\Lambda)$-bimodule such that for every finite-dimensional $\Lambda$-module $L$ the $A$-module $\shM\otimes _{\Lambda } L$ is CM and $\shM$ is flat over $\Lambda$) such that the following conditions
hold:
\begin{itemize}
\item[(i)] For every indecomposable $\Lambda$-module $L$ the $A$-module $\shM\otimes _{\Lambda } L$ is indecomposable.
\item[(ii)] If $\shM\otimes _{\Lambda } L \cong \shM\otimes _{\Lambda } L'$
 for some finite dimensional $\Lambda $-modules
$L$ and $L'$, then $L\cong L'$.
\end{itemize}

It is not difficult to check that if $X$ is algebraically wild then it is also
geometrically wild. It is not known though conjectured whether the converse is true, i.e. if geometrically wild implies algebraically wild.
\end{remark}

\vskip 2mm
The problem of classifying ACM varieties  according to the complexity of the category
of ACM sheaves that they support has recently  attracted much attention and, in particular, we would like to know whether the trichotomy finite, tame and wild representation type is exhaustive.

\vskip 2mm
Notwithstanding the fact that ACM varieties of finite representation type have been
completely classified into a short list: $\PP^n$, three or less reduced points on $\PP^2$, a smooth hyperquadric, a cubic scroll, the Veronese surface  or a rational normal curve (cf. \cite{BGS}; Theorem C and \cite{EH}; pg. 348);
 it remains open to find out the representation type of the remaining ones.
As examples of a variety of tame representation type we have the elliptic curves and
the quadric cone in $\PP^3$ (cf. \cite{CaH}; Proposition 6.1).  Quite a lot of examples of varieties of wild representation type are known (cf. \cite{To}) but so far only few examples of varieties of Ulrich wild representation type are known. To our knowledge the only examples of varieties of arbitrary dimension and of Ulrich wild representation type are the Segre varieties and the rational normal scrolls other than the quadric in $\PP^3$, the cubic scroll in $\PP^4$ and the quartic scroll in $\PP^5$ (cf. \cite{MR2013} and \cite{CMP}). In \cite{MR2014}, the second author of this paper proved that
if $X\subset \PP^n$ is a smooth ACM variety then  the
restriction
 $\nu _{3|X}$ to $X$ of the Veronese 3-uple embedding
$\nu _3:\PP^n \longrightarrow \PP^{{n+3\choose 3}-1}$ embeds $X$ as a variety
of wild representation type.
Once this paper was finished we become aware of the preprint \cite{FPL} where the authors prove that all projective ACM integral varieties other than cones are of wild CM type except for a few cases that they classify.

\vskip 2mm
In the following, we are going to
prove that
  a general  standard determinantal scheme defined by the maximal minors of a $t\times (t+c-1)$ matrix $\cA$  is of  wild representation type provided the degree matrix associated to $\cA$ satisfies  weak numerical hypothesis.  It is worthwhile to point out that our construction  proves the geometrical wildness of certain standard determinantal schemes $X\subset \PP^n$ (even more, it proves strict wildness, see Definition 2 in \cite{FPL}) but we do not prove  whether they  also are algebraically  wild. As the referee pointed out us the algebraic wildness follows from  \cite{FPL}.

\vskip 2mm
In this section, $X\subset \PP^n$ will be a general standard determinantal
scheme of codimension $c$, $\cA $ the $t \times (t+c-1)$ homogeneous matrix
associated to $X$, $I=I_t(\cA)$, $A=R/I$,$$\varphi:F:=\bigoplus _{j=1}^{t+c-1}
R(-a_j)\longrightarrow G:=\bigoplus _{i=1}^tR(-b_{i})$$ the morphism of free
$R$-modules associated to $\cA$ and $M:=\coker(\varphi )$. Without lost of generality we can assume
\begin{equation}\label{hypothesis0}b_t\le \cdots \le b_1 \text{ and } a_{t+c-1}\le \cdots \le a_2\le a_1.\end{equation}
We set $\ell :=\sum _{i=1}^{t+c-1}a_i-\sum _{j=1}^tb_j$. We have (Proposition \ref{resol} (iii)): $$K_X(n+1)\cong S_{c-1}M(\ell).$$
We will assume
$t>1$ ($X$ is a codimension $c$ complete
intersection if $t=1$) and $c\ge 1$.
 If $t\ge 2$,
$\widetilde{M}$ is a locally free $\odi{X}$-module of rank 1 over the open set
$U:=X\setminus Z$ where $I(Z):=I_{t-1}(\cA)$ and $U\hookrightarrow \PP^n$ is a
local complete intersection. Recall also that if $a_j > b_i$ for any $i,j$,
then $V(I(Z))=Sing(X)$ (cf. \cite{Ba} if the characteristic of the ground field is zero and \cite{W} for the case of arbitrary characteristic), $\codim _X(Sing(X))=c+2$ and $\codim _{\PP^n}X=c$ for a
general choice of $\varphi \in \Hom (F,G)$.

\begin{remark}\rm \label{pamsrem}
(i) Using Remark 2.1 of \cite{KM2011} with $\alpha =2$ we get $\codim _X (Sing(X))\ge 3$, whence $\depth _{I(Z)}A\ge min(3, \dim X+1)$ provided $a_{c+i}\ge b_i$ (resp. $a_{c+i}>b_i$), $1\le i \le t-1$, for $c\ge 2$ (resp. $c=1$).

(ii) Notice that for a general choice of the entries of $\cA$ all assumptions of Proposition \ref{resol} (iv) are fulfilled. In fact, let $\cB$ be the homogeneous matrix obtained deleting the last row of $\cA$. Since $\cA$ is general, $\cB$ is also general and we have $\depth_{I_t(\cA)}R=c$ and $c+1=\depth_{I_{t-1}(\cB)}R\le \depth_{I_{t-1}(\cA)}R$. Therefore, Proposition \ref{resol} (iv) applies and we conclude that $\pd S_{c+1}M=c+1.$
\end{remark}

Indeed Remark 3.3(i) is more generally stated in \cite{KMR2005} where we
shortly indicate a proof using Chang's filtered Bertini-type theorem, see
\cite{W}; Theorem 3 which also covers the characteristic p case. Since we in
\cite{KMR2005} assumed that $\cA$ was minimal, \cite{KMR2005}; Remark 2.7
applies provided $a_j \ne b_i$ for all $j,i$. Using e.g. \cite{Kle2015};
Proposition 5.1 which describes a general element of $W(b;a)$ when $a_j  =
b_i$ for some $j,i$, and induction on the number of times an equality  $a_j  =
b_i$ occur, it is a rather easy exercise to see that \cite{KMR2005}; Remark
2.7 holds also for a non-minimal matrix $\cA$, i.e. that we may skip the
assumption $a_j \ne b_i$ for all $j,i$ above when $c \ge 2$.

\vskip 2mm
As a main tool for constructing families of ACM sheaves on $X$ of  arbitrary high rank, we will use iterated extensions of ACM sheaves on $X$ of lower rank. So, let us start with a series of technical results that will allow us to ensure the existence of such extensions.

\begin{lemma}\label{ExtSiMSjMExtS2i-j-1+cMR} We keep the above notation and we assume  $\depth _{I(Z)}A\ge 3$. For any $0\le j \le c$ and $i=0,1$, we have
\begin{equation}
\Ext^{i}_{\odi{X}}(\widetilde{S_jM},\widetilde{M}^{\vee }(\mu))\cong \  _0\! \Ext^{i}_A(S_jM,M^{\vee}(\mu))\cong \  _0\! \Ext^{i+c}_R(S_{j+c}M,R(\mu -\ell)).
\end{equation}
\end{lemma}
\begin{proof} Since $S_{-1}M \cong M^{\vee}$ is MCM by Proposition \ref{resol} and $\widetilde{S_jM}\arrowvert_{X\setminus Z}$ is invertible, (\ref{NM}) with $r=2$ yields
$$\  _0\! \Ext^{i}_A(S_jM,M^{\vee}(\mu))\cong  \Ext^{i}_{\odi{X}}(\widetilde{S_jM},\widetilde{M}^{\vee }(\mu))$$
and
$$\begin{array}{rcl} \Ext^{i}_{\odi{X}}(\widetilde{S_jM},\widetilde{M}^{\vee }(\mu)) & \cong & \Hl^{i}(U,{\mathcal H}om_{\odi{X}}(\widetilde{S_jM},\widetilde{M}^{\vee }(\mu))) \\
  & \cong &
  \Hl^{i}(U,{\mathcal H}om_{\odi{X}}(\widetilde{S_{j+1}M}\otimes \widetilde{S_{c-1}M}(-\mu),\widetilde{S_{c-1}M})) \\
  & \cong & \Hl^{i}(U,{\mathcal H}om_{\odi{X}}(\widetilde{S_{c+j}M}(\ell -\mu),\widetilde{K_A}(n+1))) \\
  & \cong &   _0\! \Ext^{i}_A(S_{j+c}M(\ell -\mu),K_A(n+1)) \\
  & \cong &   _0\! \Ext^{i+c}_R(S_{j+c}M(\ell -\mu),R) \end{array}$$
for $i=0,1$ where the second last isomorphism is due to  (\ref{NM}) and the last isomorphism follows from local duality used twice (both for $A$ and $R$), or from a well known property of canonical modules since the canonical module of $R$ is $K_R=R(-n-1)$.
\end{proof}

\begin{proposition}\label{HomM-Mv} We keep the above notation and we assume that $\depth _{I(Z)}A\ge 3$ and $\sum _{j=c+1}^{t+c-1}a_j>1+\sum_{i=1}^{t-1}b_i+b_1-b_t$.  For $\mu =t+1-\sum _{j=c+1}^{t+c-1}a_j+\sum_{i=1}^{t}b_i+b_1$ we have:

\begin{itemize}
\item[(1)] $\Hom_{\odi{X}}(\widetilde{M}^{\vee }(t-\mu),\widetilde{M})=0$;
\item[(2)] $\Hom_{\odi{X}}(\widetilde{M},\widetilde{M}^{\vee }(t-\mu))=0$;
\item[(3)] $ \Ext^1_{\odi{X}}(\widetilde{M},\widetilde{M}^{\vee }(t-\mu))\cong (\wedge ^{c+1}F^{*})(-b_1-1-\sum_{j=1}^ca_j)_0$.
\end{itemize}
\end{proposition}
\begin{proof} (1) As in Lemma \ref{ExtSiMSjMExtS2i-j-1+cMR}, we get that
$\  _0\!\Hom(\widetilde{M}^{\vee }(t-\mu),\widetilde{M})\cong S_2M(\mu-t)_0$. Since
$$ \cdots \longrightarrow S_2G=\oplus _{1\le i, j\le t}R(-b_i-b_j)\longrightarrow S_2M\longrightarrow 0$$
is exact we get $\  _0\!\Hom(\widetilde{M}^{\vee }(t-\mu),\widetilde{M})=0$ provided $\mu-t-b_i-b_j<0$ for any $1\le i ,j\le t$. Since $b_1\ge \cdots \ge b_t$, this means $t>\mu -2b_t$. But $t>\mu -2b_t=t+1-\sum _{j=c+1}^{t+c-1}a_j+\sum_{i=1}^{t}b_i+b_1-2b_t$  is true  by the hypothesis $\sum _{j=c+1}^{t+c-1}a_j>1+\sum_{i=1}^{t-1}b_i+b_1-b_t$.

(2) By Lemma \ref{ExtSiMSjMExtS2i-j-1+cMR} we have
$$\Ext^{i}_{\odi{X}}(\widetilde{M},\widetilde{M}^{\vee }(t-\mu))\cong \  _0\! \Ext^{i+c}_R(S_{1+c}M,R(t-\mu -\ell)) \text{ for } i=0,1.$$
We compute the last $\Ext $ using the projective resolution of $S_{c+1}M$:
$$ 0\longrightarrow  \wedge ^{c+1}F \longrightarrow \wedge ^cF\otimes G \longrightarrow \cdots   \longrightarrow S_{c+1}G \longrightarrow S_{c+1}M\longrightarrow 0$$
given in \cite{eise} (see also Proposition \ref{resol} and Remark \ref{pamsrem}(ii)), and we get
the exact sequences

$$\longrightarrow \ _0\!\Hom(\wedge ^cF\otimes G,R(t-\mu-\ell)) \stackrel{e }{\longrightarrow} \
_0\!\Hom(\wedge ^{c+1}F,R(t-\mu-\ell))  \longrightarrow \  _0\!\Ext^{c+1}_R(S_{1+c}M,R(t-\mu -\ell))\longrightarrow 0,$$
and
$$ _0\!\Hom(\wedge ^{c-1}F\otimes S_2G,R(t-\mu-\ell)) \longrightarrow \ker (e)\longrightarrow \  _0\!\Ext^{c}_R(S_{1+c}M,R(t-\mu -\ell)) \longrightarrow 0.$$
Using our hypothesis on $\mu $, we obtain
$$ _0\!\Hom(\wedge ^cF\otimes G,R(t-\mu-\ell))  \cong  (\wedge ^cF^*\otimes G^*(t-\mu-\ell))_0
 \cong $$ $$
 \bigoplus_{1\le s_1<\cdots <s_c\le t+c-1 \atop 1\le j\le t}R((\sum _{i=1}^ca_{s_i}) +b_j+t-\mu-\ell)_0 \\
 =  0.  $$
The last equality is true because substituting $\mu$ and taking into account that $b_t\le \cdots \le b_1$ and $
 a_{t+c-1}\le \cdots \le a_2\le a_1$, we get $\sum _{i=1}^ca_{s_i} +b_j+t-\mu-\ell=\sum _{i=1}^ca_{s_i} +b_j-1-b_1-\sum_{i=1}^ca_i<0$.

(3) We also have $$ \  _0\!\Ext^{c+1}_R(S_{1+c}M,R(t-\mu -\ell))  \cong  (\wedge ^{c+1}F^*)(t-\mu-\ell))_0=(\wedge ^{c+1}F^{*})(-b_1-1-\sum_{j=1}^ca_j)_0 .$$
\end{proof}

\begin{corollary} \label{corlin} Let $X\subset \PP^n$ be a general linear standard determinantal variety of codimension $c\ge 1$. Assume $n-c\ge 2$ and $t\ge 3$. Then, we have:
\begin{itemize}
\item[(1)] $\Hom_{\odi{X}}(\widetilde{M}^{\vee }(t-2),\widetilde{M})=0$;
\item[(2)] $\Hom_{\odi{X}}(\widetilde{M},\widetilde{M}^{\vee }(t-2))=0$;
\item[(3)] $ \Ext^1_{\odi{X}}(\widetilde{M},\widetilde{M}^{\vee }(t-2))\cong (\wedge ^{c+1}F^{*})(-1-c)_0={t+c-1\choose c+1}\ge 3$.
\end{itemize}
\end{corollary}
\begin{proof} It follows from Proposition \ref{HomM-Mv} taking $a_i=1$ for all $i$ and $b_j=0$ for all $j$ and, hence, $\mu = 2$.
\end{proof}

\begin{remark} \label{rem36} \rm (i) Note that if we assume $a_j>b_i$ for any $i,j$  (or, equivalently, $a_{t+c-1}>b_1$) then the hypothesis $\sum _{j=c+1}^{t+c-1}a_j>1+\sum_{i=1}^{t-1}b_i+b_1-b_t$
 is a  very weak assumption.  Indeed, since  $a_j\ge b_i+1$ for any $j,i$,  we have
$$\sum _{j=c+1}^{t+c-1}a_j\ge b_1+\sum_{i=1}^{t-2}b_i+(t-1).$$
Hence, if we assume $t>2+b_{t-1}-b_t$, we get $\sum _{j=c+1}^{t+c-1}a_j>1+\sum_{i=1}^{t-1}b_i+b_1-b_t$ because
$$b_1+\sum_{i=1}^{t-2}b_i+(t-1)>b_1+\sum_{i=1}^{t-2}b_i+1+b_{t-1}-b_t.$$

(ii) Assuming only $a_{c+i}\ge b_i$ for $1\le i \le t-1$ (cf. Remark \ref{pamsrem} (i)) and that at least $2+b_1-b_t$ of the inequalities are strict (or otherwise increase the sum $\sum _{j=c+1}^{t+c-1} a_j$ by $2+b_1-b_t$ compared to $\sum _{i=1}^{t-1}b_i$, e.g. for $t=2$, assume $a_{c+1}\ge 2+2b_1-b_2$) we easily prove that  $\sum _{j=c+1}^{t+c-1}a_j>1+\sum_{i=1}^{t-1}b_i+b_1-b_t$.
\end{remark}

By paying a little more extra attention to the cases $t=2$, $t=3$ and $t\ge 4$ we get the following:

\begin{corollary} \label{ge3} We keep the above notation, we set $\mu :=t+1-\sum _{j=c+1}^{t+c-1}a_j+\sum_{i=1}^{t}b_i+b_1$ and we assume that $\depth _{I(Z)}A\ge 3$ and $\sum _{j=c+1}^{t+c-1}a_j>1+\sum_{i=1}^{t-1}b_i+b_1-b_t$. If $t=3$ and $a_{c+1}\le 1+b_1$ we
 also assume $a_{c+2} > b_1$ and $a_c=a_{c+1}$; and if
  $t\ge 4$ and $a_{c+1}\le 1+b_1$ we  assume  $a_{c+3}>b_1$ . Then, it holds:
\begin{itemize}
\item[(1)] $\  _0\!\Hom(M^{\vee }(t-\mu),M)=\  _0\!\Hom(M,M^{\vee }(t-\mu))=0$;
\item[(2)] $\  \dim \,  _0\! \Ext^1_{A}(M,M^{\vee }(t-\mu))\ge 3$.
\end{itemize}

\end{corollary}
\begin{proof} (1) It follows from Proposition \ref{HomM-Mv}.

(2) By Proposition \ref{HomM-Mv}(3), we have  $$ \Ext^1_{A}(M,M^{\vee }(t-\mu))=
\bigoplus_{1\le s_1<\cdots <s_{c+1}\le t+c-1 }R((\sum _{i=1}^{c+1}a_{s_i})-b_1-1-\sum_{j=1}^ca_j)$$
and we will carefully analyze how many non-zero summands we have. We distinguish 3 cases:

\vskip 2mm
\noindent \underline{Case a:} $t=2$. The assumption $\sum _{j=c+1}^{t+c-1}a_j>1+\sum_{i=1}^{t-1}b_i+b_1-b_t$ becomes $a_{c+1}>1+2b_1-b_2$  and we get $a_{c+1}>1+b_1$. By Proposition \ref{HomM-Mv}(3) with $\mu =3-a_{c+1}+2b_1+b_2$, we have
$$\  \dim \,  _0\! \Ext^1_{A}(M,M^{\vee }(t-\mu))=\dim R(a_{c+1}-b_1-1)_0\ge \dim R_1=n+1\ge 3.$$

\vskip 2mm
\noindent \underline{Case b:} $t=3$. If we take  $\mu =4-(a_{c+1}+a_{c+2})+2b_1+b_2+b_3$, we get
$$\  _0\!\Ext^{1}_A(M,M^{\vee}(t-\mu)) =  \bigoplus_{j=1}^{c+2}R(-b_1-1+a_{c+1}+a_{c+2}-a_j)_{0}.$$
 If $a_{c+1}>b_1+1$ one of the direct summands is $R_s$ for $s\ge 1$ of dimension at least $n+1$. If  $a_{c+1}\le 1+b_1$, we have $1+b_1\ge a_{c+1}\ge a_{c+2}\ge 1+b_1$ and $1+b_1=a_{c+1}=a_{c+2}=a_{c}$ by hypothesis. Therefore at least 3 direct summands are $R_0$ and we get $\  \dim \, _0\! \Ext^1_{A}(M,M^{\vee }(t-\mu))\ge 3$.
\vskip 2mm
\noindent \underline{Case c:} $t\ge 4$. By Proposition \ref{HomM-Mv}(3) we have $$\  _0\!\Ext^{1}_A(M,M^{\vee}(t-\mu)) =  \bigoplus_{1\le s_1<\cdots <s_{c+1}\le t+c+1 }R((\sum _{i=1}^{c+1}a_{s_i})-b_1-1-\sum_{j=1}^ca_j)_{0}.$$  Again if $a_{c+1}>b_1+1$ one of the direct summands is $R_s$ for $s\ge 1$ of dimension at least $n+1$.  If  $a_{c+1}\le 1+b_1$, we have $1+b_1\ge a_{c+1}\ge a_{c+2}\ge a_{c+3}\ge 1+b_1$, i.e.  $1+b_1=a_{c+1}=a_{c+2}=a_{c+3}$. It follows that at least 3 direct summands are $R_0$, hence $\  \dim \, _0\! \Ext^1_{A}(M,M^{\vee }(t-\mu))\ge 3$ and we are done.
\end{proof}

As we have seen in Corollary  \ref{corlin} and Remark  \ref{rem36} the assumption on $a_j$ and $b_i$ in Proposition  \ref{HomM-Mv}  is often fulfilled for $t\ge 3$ while there are several interesting cases where it may fail for $t=2$ and $a_{c+1}\le 1+b_1$ or more generally for $t=2$ and $a_{c+1}\le 1+2b_1-b_2$. We therefore want to improve upon Proposition \ref{HomM-Mv}  in the case
$t=2$ and $a_{c+1}= 1+2b_1-b_2$ by using its proof with $\mu _1:=\mu -1=2-a_{c+1}+2b_1+b_2$. Indeed,  we have

\begin{proposition} \label{t=2}
Let $t=2$, $\mu _1=2-a_{c+1}+2b_1+b_2$ and assume $\depth _{I(Z)}A\ge 3$ and $a_{c+1}= 1+2b_1-b_2$. Define $\alpha :=   \# \{i/ a_i=a_{c+1} \}$ and suppose $2\alpha \le n-2$ in the case $a_{c+1}=1+b_1$. Then, we have:
\begin{itemize}
\item[(1)] $\  _0\!\Hom(M^{\vee }(2-\mu _1),M)=0$;
\item[(2)] $\  _0\!\Hom(M,M^{\vee }(2-\mu_1))=0$;
\item[(3)] $\  \dim  \, _0\! \Ext^1_{A}(M,M^{\vee }(2-\mu_1))\ge 3$.
\end{itemize}
\end{proposition}

\begin{proof} We follow the proof of Proposition  \ref{HomM-Mv}.

(1) We have $\  _0\!\Hom(M^{\vee }(2-\mu _1),M)=(S_2M)_{\mu _1-2}=0$ because $\mu _1-2-2b_2=-1<0$.

(2) and (3).  The exact sequences involving $_0\! \Ext^{c+i}$ in the proof of Proposition  \ref{HomM-Mv} yields
\begin{equation} \label{map-e} \longrightarrow (\wedge ^cF^*\otimes G^*)_{\nu
  }\stackrel{e }{\longrightarrow} (\wedge ^{c+1}F^*)_{\nu } \longrightarrow \
  _0\!\Ext^{c+1}_R(S_{1+c}M,R(\nu))\longrightarrow 0 \end{equation} and a
surjection $\ker(e) \twoheadrightarrow \ _0\!\Ext^{c}_R(S_{1+c}M,R(\nu))$
where $\nu:=2-\mu _1-\ell=-b_1-\sum_{j=1}^ca_j$, $(\wedge ^{c+1}F^*)_{\nu }\cong R(a_{c+1}-b_1)_0$ and
\begin{equation} \label{wedge-c} (\wedge ^cF^*\otimes G^*)_{\nu } \cong \bigoplus _{1\le s\le c+1 \atop 1\le j \le 2}R(\sum _{i=1}^{c+1} a_i-a_s+b_j+\nu)_0=\bigoplus _{1\le s\le c+1 \atop 1\le j\le 2}R((a_{c+1}-a_s)+(b_j-b_1))_0.
\end{equation}
Since $(a_{c+1}-a_s)+(b_j-b_1)$ is zero if and only if $a_s=a_{c+1}$ and
$b_j=b_1$, we deduce that the number of direct summands equal to $R_0$ in
(\ref{wedge-c}) is exactly $\alpha \beta$ where
$\beta := \# \{j/ b_j=b_{1} \}$.

To see that
$e:(\wedge ^cF^*\otimes G^*)_{\nu }\longrightarrow (\wedge ^{c+1}F^*)_{\nu }$
has maximal rank, we use again the pairing
$$\wedge ^{i}F^*\times \wedge ^{t+c-1-i}F^* \longrightarrow \wedge
^{t+c-1}F^*\cong R(\sum _{i=1}^{t+c-1}a_i)$$
and describe the corresponding map
\begin{equation} \label{delta}
\delta _t:\wedge ^{t-1}F(\sum_{j=1}^{t+c-1}a_j)\otimes G^*({\nu}) \longrightarrow \wedge^{t-2} F(\sum _{j=c+1}^{t+c-1}a_{j}-b_1)
\end{equation}
where the restriction, $\delta _t|_0$, of $\delta _t$ to the degree zero parts
of the free modules is $e$. Here, of course, $t=2$, but to see the description
more clearly, we keep $t$ in (\ref{delta}). Let
${\mathcal A}=[\varphi(y_1),\varphi(y_2),\cdots, \varphi(y_{t+c-1})]$ where
$\{y_1,y_2,\cdots y_{t+c-1}\}$ is the standard basis of $F$. Then $\delta _t$
is given by
$$\delta _t(g^*\otimes y_1\wedge \cdots \wedge y_i \wedge \cdots \wedge y_{t-1})=\sum _{i=1}^{t-1}(-1)^{i-1}g^*(\varphi(y_i))y_1\wedge \cdots \wedge \hat{y_i} \wedge \cdots \wedge y_{t-1}.$$
If $\{g_1, g_2, \cdots ,g_t \}$ is the standard basis of $G$, then
$g_j^*(\varphi(y_i))$ is just the $j$-th coordinate of the column
$\varphi(y_i)$ (see \cite{eise}; Appendix A or the explicit description in
\cite{Ki}). For $t=2$ then $\delta _2(g_j^*\otimes y_i)=g_j^*(\varphi(y_i))$,
and if we use
${\mathcal B}=\{y_1\otimes g_1^*,y_1\otimes g_2^*,y_2\otimes g_1^*,y_2\otimes
g_2^*,\cdots ,y_{c+1}\otimes g_1^*,y_{c+1}\otimes g_2^*\}$
as an $R$-basis for $F(\sum_{j=1}^{c+1}a_j)\otimes G^*$, a matrix of
$\delta _2$ is 
given by the following $1\times (2c+2)$ matrix
\begin{equation}\label{matrix}
{\mathcal A'}:=[\varphi(y_1)^{tr},\varphi(y_2)^{tr},\cdots , \varphi(y_{c+1})^{tr}]
\end{equation}
where $\varphi(y_i)^{tr}$ is the transpose of the column $\varphi(y_i)$,
whence ${\mathcal A'}$ is obtained from the $2 \times (c+1)$ matrix
${\mathcal A}$ by converting each column $\varphi(y_i)$ of ${\mathcal A}$ to
the row $\varphi(y_i)^{tr}=[g_1^*(\varphi(y_i)),g_2^*(\varphi(y_i))]$
consisting of 2 forms of degree $(a_i-b_1,a_i-b_2)$.

Suppose $a_{c+1}\le 1+b_1$. Then $1+2b_1-b_2\le 1+b_1$ implies $b_2=b_1$ and
$a_{c+1}=1+b_1$, and the number of direct summands equal to $R_0$ in
(\ref{wedge-c}) is $2\alpha$ because $\beta=2$. Restricting $\delta_t$ in
\eqref{delta} to degree zero and removing direct ``summands'' $R_{s}$ with
$s<0$ there, we get the map $\delta_t|_0: R_0^{2\alpha } \rightarrow R(1)_0$
whose matrix is the following $1\times 2\alpha$ submatrix of ${\mathcal A'}$:
\begin{equation}\label{matrix2}
[\varphi(y_{c+2-\alpha})^{tr},\cdots, \varphi(y_c)^{tr}, \varphi(y_{c+1})^{tr}]
\end{equation}
consisting entirely of linear forms. Then using
${\mathcal B_1}=\{y_{c+2-\alpha}\otimes g_1^*,y_{c+2-\alpha}\otimes
g_2^*,\cdots ,y_{c+1}\otimes g_1^*,y_{c+1}\otimes g_2^*\}$
and ${\mathcal C}=\{x_0,x_1,\cdots ,x_n\}$ as $K$-basis of
$R_0^{2\alpha }\cong K^{2\alpha }$ and $R_1\cong K^{n+1}$ respectively, the
matrix of
\begin{equation}\label{e}
  e: R_0^{2\alpha }\cong K^{2\alpha } \longrightarrow R_1\cong K^{n+1}
\end{equation}
relative to the basis ${\mathcal B_1}$ and ${\mathcal C}$ is just the
following $(n+1)\times 2\alpha$ matrix of entries in $K$;
\begin{equation}\label{matrix3} [g_1^*(\varphi(y_{c+2-\alpha}))^{\mathcal
    C},g_2^*(\varphi(y_{c+2-\alpha}))^{\mathcal C}, \cdots,
   g_1^*(\varphi(y_{c}))^{\mathcal
    C},g_2^*(\varphi(y_{c}))^{\mathcal C}, g_1^*(\varphi(y_{c+1}))^{\mathcal
    C},g_2^*(\varphi(y_{c+1}))^{\mathcal C}]
\end{equation}
where the column $g_j^*(\varphi(y_{i}))^{\mathcal C}$ is the coordinate vector
of the linear form $g_j^*(\varphi(y_i))$ with respect to ${\mathcal C}$.
Noting that the matrix ${\mathcal A}$, hence also ${\mathcal A'}$ is given by
a random choice of homogeneous polynomials of degree $a_j-b_i$, it is clear
from (\ref{matrix2}) and (\ref{matrix3}) that
the map $e$ considered as a linear map of $K$-vector spaces, is of maximal
rank. This proves what we want because \eqref{e} and the assumption
$n+1-2\alpha \ge 3$ imply
$0=\ker (e)\twoheadrightarrow \ _0\!\Hom(M,M^{\vee }(2-\mu_1))$ as well as
$\dim \coker(e)\ge 3$, and we are done in the case $a_{c+1}\le b_1+1$.

Finally, we suppose $a_{c+1}=b_1+\nu$, $\nu\ge 2$. Since $1+2b_1-b_2=\nu+b_1$
we get $b_1-b_2=\nu -1>0$, e.g. $a_{c+1}-b_2>\nu$ and the number of $R_0$'s in
(\ref{wedge-c}) is only $\alpha $. In this case the map $e$ of (\ref{map-e})
is of the form $e:R_0^{\alpha}\longrightarrow R_{\nu}$, cf. (\ref{e}) and the
description leading to (\ref{matrix2}) shows that the $1 \times \alpha$ matrix
that corresponds to $e$ is
${\mathcal A''}:=[\varphi _1(y_{c+2-\alpha}),\cdots, \varphi_1(y_c), \varphi
_1(y_{c+1})]$
where $\varphi (y_{i})^{tr}=[\varphi _1(y_{i}),\varphi _2(y_{i})]$. If
${\mathcal D}=\{x_0^{\nu},x_0^{\nu-1}x_1,\cdots ,x_n^{\nu}\}$ is a $K$-basis
of $R_{\nu} \cong K^{n+ \nu \choose \nu}$ and
${\mathcal E}:=[\varphi _1(y_{c+2-\alpha})^{\mathcal D},\cdots,
\varphi_1(y_c)^{\mathcal D}, \varphi _1(y_{c+1})^{\mathcal D}]$
the corresponding 
matrix of $e$ with entries in $K$, we see that a general ${\mathcal A}$, which
has ${\mathcal A''}$ as a submatrix, leads to a matrix ${\mathcal E}$ with
randomly chosen entries. Hence ${\mathcal E}$ has maximal rank. Since
$\alpha \le c+1\le n-1$, we get that
${n+ \nu \choose \nu} \ge {n+2 \choose 2} \ge n+2 \ge \alpha + 3$ for any
$\nu \ge 2$, whence $\ker(e)=0$,
$\dim \coker(e) = {n+ \nu \choose \nu} -\alpha \ge 3$ and we are done.
\end{proof}

We can also treat another case not covered by Corollary \ref{ge3} (i.e. $t=3$ and $a_c>a_{c+1}=1+b_1$) by the ideas in the proof of Proposition \ref{t=2}. Note that if $t=3$, $a_{c+1}>b_1+1$ and $a_{c+2}>b_1$, then $b_3\ge b_2-1$ implies the inequality $\sum _{j=c+1}^{t+c-1}a_j>1+\sum_{i=1}^{t-1}b_i+b_1-b_t$ and Corollary \ref{ge3} applies. For $t=3$ and $a_{c+1}\le b_1+1$ (i.e. $a_{c+1}=b_1+1$ provided $a_{c+2}>b_1$) we have

\begin{proposition} \label{t=3} Let $t=3$, $\mu _1=3-a_{c+1}-a_{c+2}+2b_1+b_2+b_3$ and assume $\depth _{I(Z)}A\ge 3$, $a_c>a_{c+1}=b_1+1$, $a_{c+2}>b_1$  and $b_2-1\le b_3$.
 Then, it holds:
\begin{itemize}
\item[(1)] $\  _0\!\Hom(M^{\vee }(3-\mu _1),M)=0$;
\item[(2)] $\  _0\!\Hom(M,M^{\vee }(3-\mu_1))=0$ and $\  \dim  \,  _0\! \Ext^1_{A}(M,M^{\vee }(3-\mu_1))\ge 3$.
\end{itemize}
\end{proposition}

\begin{proof}
(1) As in the proof of Proposition \ref{t=2} we have $\  _0\!\Hom(M^{\vee }(3-\mu _1),M)=(S_2M)_{\mu _1-3}=0$ because $\mu _1-3-2b_3=-2+b_2-b_3<0$.

(2) Using the exact sequence (\ref{map-e}) with $\nu = 3-\mu _1-\ell=-b_1-\sum _{j=1}^ca_j$ and the pairing $\wedge ^{i}F^*\otimes \wedge  ^{c+2-i}F^*\longrightarrow \wedge ^{c+2}F^*\cong R(\sum _{j=1}^{c+2}a_j)$, the morphism $e$ of (\ref{map-e}) becomes
$$\begin{array}{rcl}
(\wedge ^cF^*\otimes G^*)_{\nu} & \cong & \oplus _{1\le i \le 3 \atop 1\le j_1<j_2 \le c+2}R(a_{c+1}+a_{c+2}-a_{j_1}-a_{j_2}+b_i-b_1)_0\cong \oplus _{1\le
i \le 3}R(b_i-b_1)_0 \\
\downarrow e  \ \ \ \ \ \ \  & & \\
(\wedge ^{c+1}F^*)_{\nu} & \cong & \oplus _{1\le j\le c+2}R(a_{c+1}+a_{c+2}-a_{j}-b_1)_0\cong \oplus _{1\le j\le c}R(b_1+2-a_j)_0\oplus R(1)_0^{\oplus 2}
\end{array}
$$
where $e$, with notations as in the proof of Proposition \ref{t=2}, is
determined by $\delta _3$ restricted to $\oplus _{1 \le i \le 3 }R(b_i-b_1)$.
This restriction is given by
\begin{equation} \label{del3} \delta _3(g_i^*\otimes y_{c+1}\wedge
  y_{c+2})=g_i^*(\varphi (y_{c+2}))y_{c+2}-g_i^*(\varphi (y_{c+1}))y_{c+1} \ ,
  \quad 1 \le i \le 3 \ .
\end{equation}
Indeed, every
$(j_1,j_2)\ne (c+1,c+2)$ of the index set $1\le j_1<j_2\le c+2$ corresponds to
direct ``summands'' of the form $R(s+b_i-b_1)_0=0$ for some $s<0$, and we may
take
$${\mathcal B}:=\{g_1^*\otimes y_{c+1}\wedge y_{c+2},g_2^*\otimes
y_{c+1}\wedge y_{c+2},g_3^*\otimes y_{c+1}\wedge y_{c+2}\}$$
as an $R$-basis of $\oplus _{1 \le i \le 3}R(b_i-b_1)$, as well as a $K$-basis
of $\oplus _{1 \le i \le 3}R(b_i-b_1)_0$ (if $b_i-b_1<0$, we will later skip
the ``summand'' $R(b_i-b_1)_0$ and shrink the $K$-basis ${\mathcal B}$
correspondingly and call it ${\mathcal B'}$). Then a $(c+2)\times 3$ matrix
$L({\mathcal A})$ of $\delta _3$ restricted to $\oplus _{i}R(b_i-b_1)$,
relative to the $R$-basis ${\mathcal B}$ and
${\mathcal C}:=\{y_1,y_2,\cdots ,y_{c+2} \}$, is by \eqref{del3} zero
everywhere in the first $c$ rows while the two lower rows are just the
transpose of the two last columns of ${\mathcal A}$ (up to sign). Skipping
columns of $L({\mathcal A})$ corresponding to $R(b_i-b_1)_0$ with $b_i<b_1$,
we get a matrix whose two lower rows consists of linear forms. Replacing the
$R$-basis ${\mathcal C}$ of length $c+2$ by the $K$-basis
${\mathcal C'} = \{y_1,\cdots,y_c, x_py_{c+1},x_qy_{c+2}\}$,
$1\le p,q \le n+1$ of length $m:=c+2(n+1)$, we get a $m\times i$ matrix for
some $i$, $1\le i \le 3$, that represents $e=\delta _3|_0$ as a $K$-linear
map, $e:K^{i}\longrightarrow K^{m}$, relative to the basis ${\mathcal B'}$ and
${\mathcal C'}$. 
A general choice of ${\mathcal A}$ will, however, lead to
a matrix $L({\mathcal A})$ whose entries of degree 1 are general. It follows
that $e$ has maximal rank, whence $e$ is injective and
$$\dim \coker (e) =\  \dim  \,  _0\! \Ext^1_{A}(M,M^{\vee }(3-\mu_1))\ge2(n+1)-i\ge  3$$
for every $n\ge 2$ and $i\le 3$ and we are done.
\end{proof}

We now outline the approach that will allow us to construct huge families of ACM sheaves (resp. Ulrich sheaves)  on standard determinantal varieties (resp. linear standard determinantal varieties).

\begin{construction}\label{construction1} \rm We keep the above notation and assuming $\depth _{I(Z)}A\ge 3$ we set $s:=\dim \Ext^1_{\odi{X}}(\widetilde{M},\widetilde{M}^{\vee }(t-\mu))$ (and replace $\mu $ by $\mu _1$ at places where we use Proposition \ref{t=2} or Proposition \ref{t=3}). We assume $s>0$ and we take  $0\ne e\in \Ext^1_{\odi{X}}(\widetilde{M},\widetilde{M}^{\vee }(t-\mu))$.
Using $e$ we construct a rank $2$ sheaf $\shE$ on $X$ sitting into the exact sequence:
\begin{equation}\label{extension11}
0\longrightarrow \widetilde{M}^{\vee }(t-\mu) \longrightarrow \shE \longrightarrow \widetilde{M} \longrightarrow 0.\end{equation}   $\shE $ is an ACM sheaf and arguing as in \cite{CMP}; Theorem 4.4, we prove that $\shE $ is simple and indecomposable.
Let $\shF $ be the irreducible family of rank $2$ indecomposable ACM  sheaves on $X$ given by the extension (\ref{extension11}). By construction $\shF \cong \PP(\Ext^1(\widetilde{M},\widetilde{M}^{\vee }(t-\mu)))$ and $\dim \shF = s-1.$
\end{construction}

\begin{proposition}\label{ext1E1E2} We keep the above notation and we consider $\shE _{i}$ the rank 2 ACM sheaves given by non-splitting extensions
\begin{equation}\label{extension2}  e_{i}: \ \
  0\longrightarrow \widetilde{M}^{\vee }(t-\mu) \longrightarrow \shE_{i} \longrightarrow \widetilde{M} \longrightarrow 0
  \ \ \text{ for } \ i=1,2.\end{equation} Assume that $\depth _{I(Z)}A\ge 3$. Then, it holds:
$$\ext^1(\shE _1,\shE _2)\ge \ext^1(\widetilde{M},\widetilde{M}^{\vee }(t-\mu)
)-2.$$
\end{proposition}

\begin{proof} We have $\depth _{I(Z)}A\ge 3$, hence $\Ext^1(\widetilde{M}^{\vee }(t-\mu),\widetilde{M}^{\vee }(t-\mu))\cong \Hl^1(U, \shH om(\widetilde{M}^{\vee },\widetilde{M}^{\vee }))=0$. Thus, applying $\Hom(-,\widetilde{M}^{\vee }(t-\mu))$ to  (\ref{extension2}), we get  for $i=1,2$ the exact sequences:

$$0\longrightarrow  \Hom(\widetilde{M},\widetilde{M}^{\vee }(t-\mu))  \longrightarrow  \Hom(\shE _{i},\widetilde{M}^{\vee }(t-\mu))  \longrightarrow $$

$$\begin{array}{cccccccc}
\rightarrow & \Hom(\widetilde{M}^{\vee }(t-\mu),\widetilde{M}^{\vee
              }(t-\mu))\cong K & \rightarrow &
                                               \Ext^1(\widetilde{M},\widetilde{M}^{\vee }(t-\mu)) & \rightarrow & \Ext^1(\shE _{i},\widetilde{M}^{\vee }(t-\mu))\rightarrow 0 .
\\    &
1 & \mapsto & e_{i} \end{array}$$
Using them together with Proposition \ref{HomM-Mv}(2), it follows that \begin{equation}\label{auxiliar} 0=\Hom(\widetilde{M},\widetilde{M}^{\vee }(t-\mu))\cong \Hom(\shE _{i},\widetilde{M}^{\vee }(t-\mu)), \ \text{ for } \ i=1,2; \text{ and} \end{equation} \begin{equation}\label{aux2}\ext^1(\shE _{i},\widetilde{M}^{\vee }(t-\mu)) = \ext^1(\widetilde{M},\widetilde{M}^{\vee }(t-\mu))-1 \ \text{ for } \ i=1,2.\end{equation}
Applying the functor $\Hom(-,\widetilde{M})$ to  (\ref{extension2})  with $i=1$ and using again Proposition \ref{HomM-Mv}, we get:
\begin{equation}\label{auxfin}
\Hom(\shE _{1},\widetilde{M})\cong \Hom(\widetilde{M},\widetilde{M})\cong K. \end{equation}
Now, we apply the functor $\Hom(\shE _1,-)$ to the exact sequence (\ref{extension2})
and we get the exact sequence:
{\small $$0\longrightarrow \Hom(\shE _1,\widetilde{M}^{\vee }(t-\mu))  \longrightarrow  \Hom(\shE _{1},\shE _2)  \longrightarrow
\Hom(\shE _1,\widetilde{M}) \longrightarrow  \Ext^1(\shE _1,\widetilde{M}^{\vee }(t-\mu))
 \longrightarrow  \Ext^1(\shE _{1},\shE _2) .$$}
Thus, using (\ref{auxiliar})-(\ref{auxfin}), we obtain
$$\begin{array}{rcl} \ext^1(\shE _{1},\shE _2)
& \ge & -\hom(\shE _1,\widetilde{M}^{\vee }(t-\mu))+\hom(\shE _{1},\shE _2)-\hom(\shE _1,\widetilde{M})+\ext^1(\shE _1,\widetilde{M}^{\vee }(t-\mu))
\\& \ge  & \ext^1(\widetilde{M},\widetilde{M}^{\vee }(t-\mu))-2.
\end{array}$$
\end{proof}

Finally, we will construct families of indecomposable ACM sheaves on a  standard determinantal scheme of arbitrarily high rank and dimension.
We will base our proof on the existence of low rank ACM sheaves and on the existence of non trivial extensions.
Indeed, recalling that $t>1$, we have

\begin{theorem} \label{Bigthm} Let $X\subset \PP^n$ be a  general standard determinantal scheme of codimension $c\ge 1$.
Assume that  $\dim X\ge 2$, $\sum _{j=c+1}^{t+c-1}a_j> 1+\sum_{i=1}^{t-1}b_i+b_1-b_t$ and that $a_{c+i}\ge b_i$ (resp. $ a_{c+i}>b_i$), $1\le i\le t-1$, for $c\ge 2$ (resp. $c=1$). If $t=3$ (resp. $t\ge 4$) and $a_{c+1}\le 1+b_1$ we also assume $a_{c+2}>b_1$ (resp. $a_{c+3}>b_1$).
  Then, $X$ has  wild representation type. This conclusion also holds if \ $t=2$ and $a_{c+1} = 1+2b_1-b_2$
  provided we in the case $a_{c+1}=1+b_1$ assume $2\alpha \le n-2$ where
  $\alpha := \# \{i \in \{1,c+1\}/ a_i=a_{c+1} \}$.
\end{theorem}

\begin{remark} \rm (1) According to Definition 2 and Corollary 1 in \cite{FPL}  we have proved a slightly stronger result. In fact, we  have proved that  with the above assumptions $X$ is strictly wild. We thank the referee for pointing out this  improvement.

(2) The wildness of $X$ also follows from the main result of \cite{FPL}. The result was obtained independently and, in our case, we also get a lower bound for the dimension $f(r)$ of the families $\shH _r$ of rank $r$ indecomposable, ACM sheaves that we build.
Indeed, it follows from the proof of Theorem \ref{Bigthm} and Propositions \ref{HomM-Mv}(3), \ref{t=2} and  \ref{t=3}, that $$\dim \shH _r\ge 3-2r+(r-1)\dim ( \bigoplus_{1\le s_1<\cdots <s_{c+1}\le t+c-1 }R((\sum _{i=1}^{c+1}a_{s_i})-b_1-1-\sum_{j=1}^ca_j)_0).$$

(3) The numerical assumptions in Theorem \ref{Bigthm}, see Corollary \ref{aGreaterbCase} for more natural assumptions,  summarize the requirements that allow us to construct huge families of MCM modules as iterated extensions.  We guess that these assumptions are not necessary and the following should be true: Let $X\subset \PP^n$ be a  general standard determinantal scheme of codimension $c\ge 1$ and dimension $n-c\ge 2$. Assume $t\ge 2$. Then, $X$ is of strictly wild representation type unless $X$  is a cubic scroll or a quartic scroll.
\end{remark}

\begin{proof}
It is enough to prove that  for any integer $q>0$,
   there exists an integer $r>q$ and a  family $\shH_r$ of non-isomorphic indecomposable ACM  sheaves $\shE$ on $X$ of rank $r>q$ and $\dim \shH_r \ge \frac{r}{2}(s-1)$ being $s:=\Ext^1_{\odi{X}}(\widetilde{M},\widetilde{M}^{\vee }(t-\mu))$ and $\mu :=t+1-\sum _{j=c+1}^{t+c-1}a_j+\sum_{i=1}^{t-1}b_i+b_1+b_t$. The hypothesis $\dim X\ge 2$ and Remark \ref{pamsrem}(i) imply $\depth _{I(Z)}A\ge 3$. Therefore all the
     hypothesis of Proposition  \ref{HomM-Mv} are satisfied  and we have $\dim \Ext^1_{\odi{X}}(\widetilde{M},\widetilde{M}^{\vee }(t-\mu))\ge 3$ (see Corollary \ref{ge3}). For the final conclusion and the case $t=3$, $a_c>a_{c+1}$
     we use Propositions \ref{t=2} and \ref{t=3} and we replace $\mu$ by $\mu_1$.

Given $q>0$, we choose an integer $p$ such that $2p>q$ and we will construct, using iterated extension, a simple (hence indecomposable) ACM sheaf $\shE$ on $X$ of rank $r:=2p$.
Notice that as an immediate consequence of  Proposition 5.1.3 of
\cite{PT}
we have $\Hom(\shE,\shE')=\Hom(\shE',\shE)=0$  for any two non-isomorphic rank 2 simple ACM sheaves $\shE$ and $\shE'$ of the family $\shF $ built in Construction
\ref{construction1}.
On the other hand, by Proposition \ref{HomM-Mv}(3),
the set of non-isomorphic classes of rank 2 ACM sheaves in $\shF$ has dimension $ s-1\ge 2$.

Set $\PP^{s-1} := \PP(\Ext^1 _{\odi{X}}(\widetilde{M},\widetilde{M}^{\vee }(t-\mu)))\cong \shF$ and
 denote by $U\subset \PP^{s-1} \times\stackrel{p)}{\cdots} \times \PP^{s-1}
 $ the non-empty open dense subset parameterizing closed points $[\shE_1, \cdots, \shE_p] \in \PP^{s-1} \times\stackrel{p)}{\cdots} \times \PP^{s-1}$ such that $\shE_i \ncong \shE_j$ for $i \neq j$.  Given $[\shE_1, \cdots, \shE_p] \in U$,  the set of sheaves $\shE_1, \cdots, \shE_p$ satisfy the hypothesis of \cite{PT}; Proposition 5.1.3 and therefore, there exists a family of rank $r$ simple sheaves $\shE$
on $X$ parameterized by
$$
\PP(\Ext^1(\shE_{p},\shE_1)) \times \cdots \times \PP(\Ext^1(\shE_{p},\shE_{p-1}))
$$
\noindent and given as extensions of the following type
$$
0\arr\oplus_{i=1}^{p-1}\shE_i\arr\shE\arr\shE_p\arr 0.
$$

Since $\shE_{i}$ are ACM sheaves, $\shE$ is also an ACM sheaf and we have
 constructed a family $\shH _r$ of non-isomorphic rank $r$ simple ACM sheaves $\shE$ on $X$ parameterized  by a projective bundle  over $U$ and its  dimension is given by the following formula
$$ \dim \shH _r  =
(p-1)\di(\PP(\Ext^1(\shE_{p},\shE_1))) + \dim U .
$$

By Proposition \ref{ext1E1E2}
  $\dim \Ext^1(\shE_{p},\shE_1)\ge s-2>0$. Hence, summing up, we conclude that
$$ \dim \shH _r  \ge
(p-1)(s-3) + p(s-1)\ge \frac{r}{2}(s-1)$$
which proves what we want.
\end{proof}

\begin{remark} \label{remBigthm} \rm The assumption of {\sf generality} of the
  determinantal scheme $X$ is not needed in the proofs of Lemma
  \ref{ExtSiMSjMExtS2i-j-1+cMR}, Proposition \ref{HomM-Mv} and Corollary
  \ref{ge3}, and since we can avoid using Remark \ref{pamsrem} (i) and (ii) by
  including a stronger assumption on $\dim X$, we get a variation of Theorem
  \ref{Bigthm} valid for every standard determinantal scheme. Recalling that
  $Z=V(I_{t-1}(\cA))$ and $\depth _{I(Z)}A = \dim A - \dim R/I(Z)$, the proof
  of Theorem \ref{Bigthm} above applies and we get the following result:

  Let $X\subset \PP^n$ be a standard determinantal scheme of codimension
  $c\ge 1$ and suppose $\dim X \ge 2 + \dim R/I(Z)$ and
  $\sum _{j=c+1}^{t+c-1}a_j> 1+\sum_{i=1}^{t-1}b_i+b_1-b_t$. Moreover if $t=3$
  (resp. $t\ge 4$) and $a_{c+1}\le 1+b_1$ we also assume $a_{c+2}>b_1$ and
  $a_{c}=a_{c+1}$ (resp. $a_{c+3}>b_1$). Then, $X$ has wild representation
  type.
\end{remark}

\begin{corollary}\label{linear case} Let $X\subset \PP^n$ be a general linear standard determinantal scheme of codimension $c\ge 1$ defined by the maximal minors of a $t\times (t+c-1)$ matrix. Assume $n-c\ge 2$ and $t>2$. Then, $X$ is of wild representation type.
\end{corollary}
\begin{proof} It immediately follows from Theorem \ref{Bigthm} taking $b_j=0$ for all $j$ and $a_i=1$ for all $i$. Notice that the hypothesis $\sum _{j=c+1}^{t+c-1}a_j>1+\sum_{i=1}^{t-1}b_i+b_1-b_t$ is now equivalent to  $t>2$.
\end{proof}

\begin{remark} \rm The cases of linear determinantal varieties not covered in the above Corollary will be treated in next section where a strong result will be obtained. Indeed, we will prove that if $n-c=1$ and $t>2$, or $t=2$ and $(n-c,c)\notin \{(2,2),(2,3) \}$, then $X$ has Ulrich wild representation type (see Corollary \ref{maincoro} and Theorem \ref{mainthm2}).
\end{remark}

The following result applies also to $t=2$.

\begin{corollary}\label{aGreaterbCase}  Let $X\subset \PP^n$ be a general
  standard determinantal scheme of codimension $c\ge 1$. Assume that $\dim
  X\ge 2$, $a_j > b_i$ for any $j,i$ and that $b_{t-1}-b_t \le \max\{0,t-3\}$.
  Moreover  if $a_{c+1}= 1+b_1$ and $t=2$ we also suppose that
  $2\alpha \le n-2$  where $\alpha := \# \{i \in
  \{1,c+1\}/ a_i=a_{c+1} \}$. Then, $X$ has wild representation type.
\end{corollary}
\begin{proof} Using Remark~\ref{rem36}(i) we get $\sum
  _{j=c+1}^{t+c-1}a_j>1+\sum_{i=1}^{t-1}b_i+b_1-b_t$ provided $b_{t-1}-b_t \le
  t-3$. Hence if $t \ge 3$, $X$ is wild by Theorem \ref{Bigthm}. For $t =2$, $X$ is wild
  by the final sentence in Theorem \ref{Bigthm}.
\end{proof}


\section{Ulrich  bundles of low rank}

From now on, we will only deal with linear standard determinantal schemes $X\subset \PP^n$ of codimension $c\ge 1$ and we will assume $t\ge 2$ since the case $t=1$ corresponds to a projective space $\PP^{n-c}\subset \PP^n$ and ACM bundles on  projective spaces were classified by Horrocks.

Recall that, in the particular case of a linear standard determinantal scheme $X\subseteq \PP^n$ associated to a $t\times (t+c-1)$ matrix, the homogeneous coordinate ring $A$ of $X$ has a linear minimal free resolution, i.e., its minimal graded free resolution has the following shape
$$0\rightarrow R(-t-c+1)^{{t+c-2 \choose t-1}}\rightarrow \cdots \rightarrow
R(-t-i+1)^{\rho _i} \rightarrow  \cdots  \rightarrow R(-t)^{{t+c-1 \choose t}} \rightarrow R\rightarrow A\rightarrow 0$$
where $$\rho _{i}={c+t-1\choose i+t-1}{i+t-2\choose t-1} \text{ for } 1\le i \le c. $$
In particular, $I(X)$ is generated by forms of degree $t$ and, by  \cite{b-v}; Proposition 2.15, we have
\begin{equation} \label{degree} \deg(X)={t+c-1\choose c}.\end{equation}

 \vskip 2mm Our aim is to address the problem whether $X$ is of wild representation type with respect to the much more restrictive category of its indecomposable Ulrich sheaves. Our approach will be analogous to the one developed in \S 3.
So, the first goal  is to characterize  all Ulrich line bundles on a linear standard determinantal variety $X$ defined by a $t\times (t+c-1)$ matrix. As an application, we will prove that, under some assumptions on $n, c$, and $ t$,  linear standard determinantal varieties have Ulrich wild representation type.  Indeed,   as iterated extensions of Ulrich  line bundles, we will construct on $X$ families of indecomposable Ulrich  bundles of arbitrary high rank and dimension.

Let us first describe the
 Picard group of a smooth linear standard determinantal scheme $X$.
 Assume that $X\subset \PP^n$ is given by the
maximal minors of a $t \times (t+c-1)$  homogeneous matrix $\cA$
representing a homomorphism $\varphi$ of free graded R-modules $$
\varphi: F= \bigoplus_{i=1}^{t+c-1} R(-1) \longrightarrow G=
\bigoplus_{j=1}^{t} R .  $$  Denote by $H$ the general hyperplane section of $X$, by $W\subset X$ the codimension 1 subscheme of $X$ defined by the maximal minors of the $(t-1)\times (t+c-1)$ matrix obtained deleting the last row of $\cA$ and by $Y\subset X$ the codimension 1 subscheme of $X$ defined by the maximal minors of the $t\times (t+c)$ matrix obtained adding a column of general linear forms to $\cA$. Observe that $Y\thicksim W+H$. The following theorem computes the Picard
group of $X$. Indeed, we have:

\begin{theorem}\label{picard}  Let $X\subset \PP^n$ be a smooth  standard linear
determinantal scheme of codimension $c\ge 2$ associated to a $t\times (t+c-1)$ matrix.  If $n-c> 2$ or $n-c=2$ and $t \ge 4$, then
 $Pic(X) \cong \mathbb{Z}^2\cong \langle H,W \rangle$.

 In addition, we have
 $$K_X\cong (t-n+c-2)H+(c-1)W\cong (t-n-1)H+(c-1)Y.$$
\end{theorem}

\begin{proof}
See \cite{Ein}; Corollary 2.4 for smooth standard determinantal varieties
$X\subset \PP^n$ of dimension $n-c\ge 3$ and \cite{FF};  Proposition 5.2 for the case $n-c=2$ (see also \cite{Lo}; Theorem III.4.2,
for smooth surfaces $X\subset \PP^4$).
\end{proof}

As an application of Proposition \ref{resol} and Theorem \ref{picard} we get an explicit list of all ACM line bundles and all Ulrich line bundles on smooth linear standard determinantal varieties. In fact, we have

\begin{proposition} \label{Ulrichlinebundles} Let $X\subset \PP^n$ be a smooth standard linear determinantal scheme of codimension $c\ge 2$ defined by the maximal minors of a $t\times (t+c-1)$ matrix $\cA$. Assume  $n-c> 2$ or $n-c=2$ and $t \ge 4$.  Let $\cL$ be a line bundle on $X$. It holds:
 \begin{itemize} \item[(i)] $\cL$ is an ACM line bundle on $X$ if and only if $\cL\cong \odi{X}(aY+bH)$ with $-1\le a\le c$ and $b\in \ZZ$;
 \item[(ii)] $\cL$ is an initialized Ulrich line bundle if and only if
   $\cL\cong \odi{X}(-Y+tH)$ or $\cL\cong \odi{X}(cY-cH)$.
\end{itemize}
\end{proposition}
\begin{proof} (i) Since $\Hl^0_*(\odi{X}(aY+bH))\cong \widetilde{S_aM(\delta)}$ for a suitable $\delta$, the results follows from \cite{eise}; Theorem A2.10.

(ii) It follows from (i) and \cite{eise}; Theorem A2.10.
\end{proof}

The existence of rank 1 Ulrich modules on any linear determinantal scheme was first proved in \cite{BHU}; Proposition 2.8,
  and to our knowledge the first example of Ulrich module of high rank on a linear determinantal variety was given in \cite{KM2012}; Theorem 3.2, where the authors prove that under some mild hypothesis  $\Ext^1(M,S_{c-1}M)$ is an Ulrich module of rank $c$. The above Proposition gives the complete list of rank 1 Ulrich bundles on smooth linear standard determinantal schemes of dimension $\ge 2$. Nevertheless,  the hypothesis smooth as well as the hypothesis $c\ge 2$ and $n-c\ge 2$ can be dropped,  and we have

\begin{proposition}\label{Ulrichmodulesrk1}
Let $X\subset \PP^n$ be a standard linear determinantal scheme of codimension $c$ defined by the maximal minors of a $t\times (t+c-1)$ matrix $\cA$. Set $R=k[x_0,\cdots,x_n]/I_t(\cA)$ and denote by $\mathfrak{p} $ (resp. $\mathfrak{q} $)
the ideal generated by the $(t-1)\times (t-1)$ minors of the first $t-1$ rows (resp. columns) of $\cA$. It holds:
\begin{itemize}
\item[(i)] $\mathfrak{p}^{\ell} $  and $\mathfrak{q} ^{\ell }$ represents all reflexive rank 1 modules
\item[(ii)] $\mathfrak{p}^{\ell} $ (resp. $\mathfrak{q} ^{\ell }$)  is a CM ideal if and only if $\ell \le 1$ (resp. $\ell \le c$).
\item[(iii)] $\mathfrak{p} $  and $\mathfrak{q} ^{c }$ are the only Ulrich ideals.
    \end{itemize}
\end{proposition}
\begin{proof} Both (i) and (ii)   follow from \cite{BRW}; Theorem 4.3.
\end{proof}

\begin{remark} \rm (a) It is worthwhile to point out that Proposition \ref{Ulrichmodulesrk1} was stated in \cite{BRW} in a more general set up. Indeed, it was stated by linear determinantal schemes defined by the $r\times r$ minors of a $m\times n$ matrix. In this paper we focus our attention in the particular case $r=min\{m,n\}$.

(b) With the previous notation, Proposition \ref{Ulrichmodulesrk1} implies:
 \begin{itemize} \item[(i)] $S_{i}M$ is a maximal Cohen-Macaulay $A$-module if and only if  $-1\le i\le c$ (we  interpret $S_{-1}M$ as $\Hom_A(M,A)$);
  \item[(ii)]  $S_{i}M$ is an   Ulrich
module if and only if $i=-1$ or $c$.
\end{itemize}
\end{remark}

\begin{notation} \label{nota} \rm From now, we set $\cL _1:= \odi{X}(-Y+tH)$ and $\cL
  _2:=\odi{X}(cY-cH)$ if $X=U$, i.e. $X$ is a local complete intersection in
  $\PP^n$. The corresponding $R/I(X)$-modules are denoted by
  $L_1=\Hl^0_*(\cL_1)$ and $L_2=\Hl^0_*(\cL_2)$. Note that $L_1=M^{\vee }(t-1)$
  and $L_2=S_cM$. If $X \ne U$, we put $\cL _1= \widetilde M^{\vee }(t-1)$,
  $\cL _2= \widetilde {S_{c}M}$ and we let $\cK_X = \widetilde
  {S_{c-1}M}(t+c-n-2)$ be the canonical sheaf.
\end{notation}

We observe that $\cL _1 \cong \cL_2^{\vee} \otimes \cK_X\otimes \odi{X} ((n-c+1)H)$ if
$X=U$. This isomorphism suggests us that maybe the fact of being Ulrich is
invariant under dualizing and twisting by suitable twist of the canonical
sheaf. In fact, we have the following generalization of Lemma 2.4 in
\cite{CH}.

\begin{lemma} Let $\shE$ be an initialized Ulrich sheaf of rank $r$ on an ACM
  projective scheme $X$ and let \ $d= \dim X$. Then $\shH om_{{\mathcal
      O}_X}(\shE,\cK_X(d+1))$ is also an initialized Ulrich sheaf of rank $r$.
  Moreover if $X \subset \PP^n$ is a local complete intersection, then we have
  an isomorphism of initialized Ulrich sheaves; \[ \shE ^{\vee} \otimes \cK_X\otimes  \odi{X}
  ((d+1)H) \simeq \shH om_{{\mathcal O}_X}(\shE,\cK_X(d+1)).\]
  \end{lemma}
\begin{proof}
  It is well known that $\shH om_{{\mathcal O}_X}(\shE,\cK_X)$ is ACM if and
  only if $\shE$ is ACM. Moreover if $\shE$ is an initialized Ulrich sheaf,
  then the minimal resolution of $\shE$ is linear with $\cO_{\PP^n}(-c)^s$ as
  its leftmost term where $s=m(\shE)$. Indeed the minimal resolution of $\shE$
  is self-dual with regard to the graded Betti numbers by \cite{ESW};
  Corollary 2.2. Since $n-c=d$, we get
  \[ \shH om_{{\mathcal O}_X}(\shE,\cK_X(d+1)) \simeq \shE xt_{{\mathcal
      O}_{\PP^n}}^c(\shE,\cO_{\PP^n}(-c))\] by using duality twice (on $X$ and
  on $\PP^n$). Applying $ \shH om_{{\mathcal O}_X}(-,\cO_{\PP^n})$ to the
  minimal resolution of $\shE (c)$, it follows that $\shE xt_{{\mathcal
      O}_X}^c(\shE(c),\cO_{\PP^n})$ is an initialized Ulrich sheaf and we get
  the lemma.
\end{proof}

Let us now move forward to the construction of  initialized indecomposable Ulrich bundles of higher rank on a linear standard determinantal variety $X$. The idea is to construct them as iterated extensions of $\cL_1$ by $\cL_2$. So, we start collecting  a series of technical lemmas and  fixing some extra notation  which will allow us to conclude that under some numerical restrictions on $n$, $c$ and $t$, we have $\Ext^1_ {\odi{X}}(\cL_2,\cL_1)\ne 0$.

\begin{lemma}\label{HomL1L2} We keep the above notation and we assume that $\depth _{I(Z)}A\ge 2$, e.g. $X$ a general determinantal scheme and  $\dim X\ge 1$. Then, it holds:
\begin{itemize}
\item[(1)] $\Hom_{\odi{X}}(\cL_1,\cL_2)\cong  \  _0\! \Hom(L_1,L_2)= S_{c+1}M(1-t)_0=0$;
\item[(2)] $\Hom_{\odi{X}}(\cL_2,\cL_1)\cong \  _0\! \Hom(L_2,L_1)=0$.
\end{itemize}
\end{lemma}
\begin{proof} Since the proof of (1) and (2) are analogous we will only prove (1) and we leave (2) to the reader. We have $\depth _{I(Z)}A\ge 2$, see Remark \ref{pamsrem} (i), hence $\depth _{\mathfrak{m}} A\ge 2$, whence by (\ref{NM}); $0=\Hom_{\odi{U}}(\cL_1,\cL_2)\cong \Hom_{\odi{X}}(\cL_1,\cL_2)\cong \,  _{0} \Hom(L_1,L_2)$
where the first equality comes from the fact that it does not exist a non-zero morphism between two different rank 1 Ulrich sheaves.
\end{proof}

 \begin{lemma}\label{ExtL2L1H} We keep the notation and we set $U:=X\setminus Z$ where $I(Z)=I_{t-1}(\cA)$. Assume $\depth _{Z}\odi{X}\ge i+2$. Then, we have
\begin{equation}
  \Ext^{i}_{\odi{X}}(\cL_2,\cL_1(\nu))\cong \Ext^{i}_{\odi{U}}(\cL_{2|U} ,\cL_{1|U}(\nu)) \cong \Hl^{i}(U, \cL_2^{\vee}\otimes\cL_1(\nu)) \cong\ \Ext^{i+c}_{\odi{\PP^n}}(\widetilde {S_{2c}M}(c),\cO_{\PP^n} (\nu)) .
\end{equation}
\end{lemma}
\begin{proof} Since $\cL_1$ is  ACM, using  the exact  sequence, cf. \cite{SGA2}; expos\'{e} VI,
$$\cdots \rightarrow \Ext^{i}_{Z}(\cL_2,\cL_1 (\mu))\rightarrow \Ext^{i}_{\odi{X}}(\cL_2,\cL_1 (\mu))\rightarrow \Ext^{i}_{\odi{U}}(\cL_{2|U},\cL_{1|U} (\mu))\rightarrow \Ext^{i+1}_{Z}(\cL_2,\cL_1 (\mu))\rightarrow \cdots
$$
 and the spectral sequence
$$ \Ext^p_{\odi{X}}(\cL_2,{\mathcal H}^q_{Z}(\cL_1 (\mu)))\Rightarrow \Ext^{p+q}_{Z}(\cL_2,\cL_1(\mu))
$$
we get the two first isomorphisms because $\depth _{Z} \cLÑ_1\ge 2+i$
implies that the sheaves ${\mathcal H}^q_{Z}(\cL_1 (\mu))=0$ for $q\le i+1$. To see the rightmost
isomorphism we use Proposition~\ref{resol} (iii), cf. the proof of Lemma \ref{ExtSiMSjMExtS2i-j-1+cMR}, and we get that $$
{\shH}om_{\odi{U}}(\cL _2, \cL_1) \cong {\shH}om_{\odi{U}}(\widetilde{S_{c}M},
\widetilde{M}^{\vee}(t-1)) \cong
{\shH}om_{\odi{U}}(\widetilde{S_{2c}M}(c-n-1), \cK_X) \ .$$ Then the
spectral sequence argument above and $\depth _{Z} \cK_X \ge 2+i$
imply the first isomorphism in $$ \Hl^{i}(U, \cL_2^{\vee}\otimes\cL_1(\nu))
\cong \Ext^{i}_{\odi{X}}(\widetilde{S_{2c}M}(c), \cK_X(n+1+\nu)) \cong
\Ext^{i+c}_{\odi{\PP^n}}(\widetilde{S_{2c}M}(c), \odi{\PP^n}(\nu))$$ while the
second isomorphism may be seen by e.g. using Serre duality twice.
\end{proof}

\begin{lemma}\label{ExtL2L1ExtS2cMR} We keep the notation and we assume  $\depth _{Z}\odi{X}\ge i+2$. Then we have
\begin{equation}
\Ext^{i}_{\odi{X}}(\cL_2,\cL_1(\nu))\cong \ _{\nu}\! \Ext^{i+c}_R(S_{2c}M(c),R) \ \ \text{provided  } \ \nu \ge -\dim X.
\end{equation}
\end{lemma}
\begin{proof}
  Let $E:=\Hl^{n+1}_{\mathfrak{m}}(R)$. As in the proof of
  Lemma~\ref{ExtL2L1H} there is an exact sequence
{\small
$$ \rightarrow \ _{\nu}\! \Ext^{i+c}_{\mathfrak{m}}(S_{2c}M(c),R)\rightarrow \  _{\nu}\! \Ext^{i+c}_{R}(S_{2c}M(c),R)\rightarrow \Ext^{i+c}_{\odi{\PP^n}}(\widetilde {S_{2c}M}(c),\cO_{\PP^n} (\nu))\rightarrow \ _{\nu}\! \Ext^{i+c+1}_{\mathfrak{m}}(S_{2c}M(c),R)\rightarrow
$$ }and a spectral sequence that degenerates and leads to
 \begin{equation}\label{speclem37} _{\nu}\!
 \Ext^{j+n+1}_{\mathfrak{m}}(S_{2c}M(c),R) \cong \ _{\nu}\!
 \Ext_R^{j}(S_{2c}M(c),E) \ .
\end{equation}
Hence if $i+c+1 < n+1$, we get by \eqref{speclem37} that the map $$\phi: \
_{\nu}\! \Ext^{i+c}_{R}(S_{2c}M(c),R)\rightarrow \
\Ext^{i+c}_{\odi{\PP^n}}(\widetilde {S_{2c}M}(c),\cO_{\PP^n} (\nu))$$ in the
long exact sequence above is an isomorphism for every $\nu$. For
$j:=(i+c+1)-(n+1) \ge 0$ we get by \eqref{speclem37} that $\phi$ is an
isomorphism for those $\nu$ satisfying $$ _{\nu} \!
\Ext^{j-1}_R(S_{2c}M(c),E)= 0 \ \ {\rm and} \ \ _{\nu} \!
\Ext^{j}_R(S_{2c}M(c),E)= 0 \ .$$ Recalling $\dim X=n-c$ it suffices to show
that $_{\nu} \! \Ext^{j}_R(S_{2c}M(c),E)= 0$ for $ \nu \ge -\dim X -j$ (because only the last $\Ext$-group must vanish for $j=0$). To
prove it we apply the functor $_{\nu} \! \Hom_R(-,E)$ to the free resolution
of $S_{2c}M(c)$ deduced from Proposition~\ref{resol} (iv), and we get the
following complex {\small
\begin{equation*}\label{lem37}
\rightarrow \ _\nu \! \Hom_R(\wedge^{j-1}F\otimes S_{2c-j+1}G(c),E)  \rightarrow \  _\nu \! \Hom_R(\wedge^{j}F\otimes S_{2c-j}G(c),E)  \rightarrow \  _\nu \! \Hom_R(\wedge^{j+1}F\otimes S_{2c-j-1}G(c),E)  \rightarrow
\end{equation*}
}whose homology group in the middle is exactly $ _\nu \!
\Ext^{j}_R(S_{2c}M(c),E)$. The whole group in the middle 
is, however, isomorphic to $E(j-c)_{\nu}^{\oplus p}$ for some $p$ and since
$E_q=0$ for $q\ge -n$ we get what we want. Combining with the last isomorphism
in Lemma~\ref{ExtL2L1H} we are done.
\end{proof}

\vskip 2mm From now, given integers $t, c\ge 2$ we set

\begin{equation}\label{fgh}\begin{array}{cclcl} f(t,c) & := & \sum _{i=0}^{t-1}(-1)^{i}{t+c-1\choose t-i-1}{t+c-1-i\choose t-1}{c+2+i\choose i} & = & \frac{(4+12c+8c^2-5ct-7c^2t-ct^2+c^2t^2)(c+t-1)!}{2(2+c)!(t-1)!} \\
g(t,c) & := & \sum _{i=1}^{t-1}(-1)^{i}{t+c-1\choose t-i-1}{t+c-1-i\choose t-1}{c+1+i\choose i-1} & = & \frac{c(-2-4c-t+ct)(c+t-1)!}{2(2+c)!(t-2)!} \\
h(t,c) & := & \sum _{i=2}^{t-1}(-1)^{i}{t+c-1\choose t-i-1}{t+c-1-i\choose t-1}{c+i\choose i-2} & = & \frac{(-c+c^2)(c+t-1)!}{2(2+c)!(t-3)!}, \  \ (h(t,c):=0 \text{ for } t=2); \end{array}\end{equation}

\vskip 2mm
\noindent and, for any  linear determinantal scheme $X\subset \PP^n$ of dimension $d=n-c\ge 2$, we define
$$\chi_d(\cL _{21})(\nu ):=\ext^0_{\odi{X}}(\cL_2,\cL_1(\nu))-\ext^1_{\odi{X}}(\cL_2,\cL_1(\nu))+\ext^2_{\odi{X}}(\cL_2,\cL_1(\nu)).$$
The simplification of the sums of products of binomials in (\ref{fgh}) are performed using the program Mathematica Wolfram.

\begin{theorem}\label{XdL21} Let $X\subset \PP^n$ be  a general linear determinantal scheme of codimension $c$ defined by the maximal minors of a $t\times (t+c-1)$ matrix $\cA$. Assume  $c, d=n-c\ge 2$. Then, we have
\begin{itemize}
\item[(1)] $\chi_d(\cL _{21})(0)\le {d-1\choose 2}h(t,c)+(d-2)g(t,c)+f(t,c)$,
\item[(2)] $\chi_d(\cL _{21})(-1)\le (d-2)h(t,c)+g(t,c)$,
\item[(3)]  $\chi_d(\cL _{21})(-2)\le h(t,c)$
\end{itemize}
and, the inequalities become equalities if $d=2$, or $2\le t\le 3$.
In particular,  $\Ext^1 _{\odi{X}}(\cL_2,\cL_1)\ne 0$ provided $\chi_d(\cL _{21})(0)<0$, e.g. $${d-1\choose 2}(c^2-c)(t-1)(t-2)+(d-2)c(ct-2-4c-t)(t-1)+(4+12c+8c^2-5ct-7c^2t-ct^2+c^2t^2)<0.$$
\end{theorem}
\begin{proof}  First of all we observe  that since $X$ is a  general linear determinantal scheme, we have $\depth _{Z} \odi{X}\ge 4$ being $I(Z)=I_{t-1}(\cA)$ and the same inequality works when we replace  $X$ by $X\cap \PP^r$ where $\PP^r$ is a general linear subspace $\PP^r\subset \PP^n$ of dimension $r\ge c+2$  (i.e. $\dim X\cap \PP^r\ge 2$). Therefore,  cutting $X$ with a general hyperplane, we get an exact sequence
$$0 \longrightarrow \odi{X}(-1)\longrightarrow \odi{X} \longrightarrow \odi{X\cap H}\longrightarrow 0$$
which induces an exact sequence
\begin{equation}\label{esequence}
\cdots \longrightarrow \Ext^{j-1}_{\odi{X\cap H}}(\cL_{2|H},\cL_{1|H}(\nu))\longrightarrow \Ext^{j}_{\odi{X}}(\cL_2 ,\cL_1 (\nu-1)) \longrightarrow \Ext^{j}_{\odi{X}}(\cL_2 ,\cL_1 (\nu))\longrightarrow \end{equation}
$$\Ext^{j}_{\odi{X\cap H}}(\cL_{2|H},\cL_{1|H}(\nu))\longrightarrow \cdots $$
for $j\le 2$. Indeed, it easily follows from Lemma \ref{ExtL2L1H}. Hence, letting
$$T_3:=\coker (\Ext^{2}_{\odi{X}}(\cL_2 ,\cL_1 (\nu))\longrightarrow \Ext^{2}_{\odi{X\cap H}}(\cL_{2 |H},\cL_{1|H}(\nu)))$$
we obtain
\begin{equation}\label{induction1} \chi_d(\cL _{21})(\nu)=\chi_d(\cL _{21})(\nu-1)+\chi_{d-1}(\cL _{21})(\nu)-\dim T_3.\end{equation}
It follows that for $d\ge 3$ we have
\begin{equation}\label{induction} \chi_d(\cL _{21})(\nu)\le \chi_d(\cL _{21})(\nu-1)+\chi_{d-1}(\cL _{21})(\nu)\end{equation}
with equality if $T_3=0$. Let us compute $\Ext^{i}_{\odi{X}}(\cL_2 ,\cL_1 (\nu-1))$ for $\nu \ge -\dim X$.  By Lemma \ref{ExtL2L1ExtS2cMR} we have
$\Ext^{i}_{\odi{X}}(\cL_2,\cL_1(\nu))\cong \  _{\nu} \! \Ext^{i+c}_R(S_{2c}M(c),R)$ for any  $ \nu \ge -\dim X$.
 Since by hypothesis $X$ is a general linear determinantal scheme, the complex associated to $S_{2c}M$ in \cite{eise}; Theorem A2.10 is a free resolution of $S_{2c}M$, see Remark \ref{pamsrem}(ii). In order to compute  $_{\nu} \! \Ext^{i+c}_R(S_{2c}M(c),R)$ we apply the functor $_\nu \! \Hom_R(-,R)$ to the resolution of $S_{2c}M(c)$. For  $-3\le \nu \le 0$, we get
{\small \begin{equation}\label{cd}
\begin{array}{cccccc}
\rightarrow _\nu \! \Hom_R(\wedge^cF\otimes S_cG(c),R) & \rightarrow &  _\nu \! \Hom_R(\wedge^{c+1}F\otimes S_{c-1}G(c),R) & \rightarrow &  _\nu \! \Hom_R(\wedge^{c+2}F\otimes S_{c-2}G(c),R) & \rightarrow \cdots \\ \\
  \downarrow \simeq & & \downarrow \simeq  & & \downarrow \simeq & \\ \\
  R_{\nu} ^{{t+c-1\choose c}^2} & & R_{\nu +1} ^{{t+c-1\choose c+1}{t+c-2\choose c-1}}& & R_{\nu +2} ^{{t+c-1\choose c+2}{t+c-3\choose c-2}}
\end{array}.
\end{equation} }
In particular,  $\Ext^{i}_{\odi{X}}(\cL_2,\cL_1(-3))=0$ for $i\le 2$ and $d\ge 3$  and the inequality (\ref{induction}) implies
$$\begin{array}{ccl}
\chi_d(\cL _{21})(-2) & \le & \chi_{d-1}(\cL _{21})(-2)\le \chi_{d-2}(\cL _{21})(-2)\le \cdots \le \chi_{2}(\cL _{21})(-2), \\
\chi_d(\cL _{21})(-1) & \le & \chi_{d}(\cL _{21})(-2)+ \chi_{d-1}(\cL _{21})(-1) \\
& \le & \chi_{d}(\cL _{21})(-2)+ \chi_{d-1}(\cL _{21})(-2)+\chi_{d-2}(\cL _{21})(-1) \\
& \le & \cdots \\
& \le & (d-2)\chi_{2}(\cL _{21})(-2)+ \chi_{2}(\cL _{21})(-1),
\end{array}
$$
and, similarly,
$$\begin{array}{ccl}
\chi_d(\cL _{21})(0) & \le & \chi_{d}(\cL _{21})(-1)+ \chi_{d-1}(\cL _{21})(0) \\
& \le & (d-2)\chi_{2}(\cL _{21})(-2)+ \chi_{2}(\cL _{21})(-1)+(d-3)\chi_{2}(\cL _{21})(-2)+\chi_{2}(\cL _{21})(-1)\\ & & +\chi_{d-2}(\cL _{21})(0)\\
& \le & \cdots
\\
& \le & {d-1\choose 2}\chi_{2}(\cL _{21})(-2)+ (d-2)\chi_{2}(\cL _{21})(-1)+ \chi_{2}(\cL _{21})(0).
\end{array}
$$
Let us now compute  $\chi_{2}(\cL _{21})(\nu)$ for $-2\le \nu \le 0$. For $n=c+2$, we have $\dim R_{\nu}=\dim K[x_0,\cdots ,x_n]_{\nu}={c+2+\nu\choose \nu}$. Moreover, the rightmost term in the full complex given by (\ref{cd}) is
\begin{equation}\label{lastterm}
_{\nu} \! \Hom_R(\wedge^{t+c-1}F\otimes S_{c+1-t}G(c),R)\cong R_{t-1+\nu}^{{t+c-1\choose t+c-1}{c\choose t-1}},
\end{equation}
whence  we use that the sum defining $f(t,c)$, $g(t,c)$ and $h(t,c)$ are precisely given as the alternating sum of the dimension of the terms in the full complex given by (\ref{cd}). Thus, it follows that
$$\chi_2(\cL _{21})(0)= f(t,c), \quad  \chi_2(\cL _{21})(-1)= g(t,c), \quad \chi_2(\cL _{21})(-2)= h(t,c) $$
since $_{\nu} \! \Ext^{c+i}_R(S_{2c}M(c),R)\cong \Ext^{i}_{\odi{X}}(\cL_2,\cL_1(\nu))=0$ for $i\ge 3$ and $\nu\ge -\dim X=-2$ by  Lemma \ref{ExtL2L1ExtS2cMR}.

Putting altogether we get  $$\chi_d(\cL _{21})(0)\le {d-1\choose 2}h(t,c)+(d-2)g(t,c)+f(t,c),$$
 $$\chi_d(\cL _{21})(-1)\le (d-2)h(t,c)+g(t,c)$$ and
  $$\chi_d(\cL _{21})(-2)\le h(t,c)$$ with equalities for $d=2$. We simplify the formulas and we get  $\Ext^1 _{\odi{X}}(\cL_2,\cL_1)\ne 0$ provided $${d-1\choose 2}(c^2-c)(t-1)(t-2)+(d-2)c(ct-2-4c-t)(t-1)+(4+12c+8c^2-5ct-7c^2t-ct^2+c^2t^2)<0.$$
Finally, it remains to prove that for $2\le t\le 3$ and $d\ge 3$ the above inequalities turns out to be equalities.  In these cases, the complex (\ref{cd}) consists of 2 or 3 terms  (cf. (\ref{lastterm}))
and $_{\nu} \! \Ext^{c+i}_R(S_{2c}M(c),R)=0$ for $i\ge 2$ (resp. $i\ge 3$) for $t=2$ (resp. $t=3$). Using once more Lemma \ref{ExtL2L1ExtS2cMR} for $-3\le \nu \le 0$, we get $T_3=0$ and we conclude by (\ref{induction}).
\end{proof}

\begin{remark}
\label{ext1(L2,L1)t=2,3}  Let $X\subset \PP^n$ be a general  linear standard determinantal scheme of codimension $c$ defined by the maximal minors of a $t\times (t+c-1)$ matrix. Assume  $c, n-c\ge 2$. It holds: \begin{itemize} \item[(i)] if $t=2$, then $\chi_d(\cL _{21})(0)= c+1-c(n-c)<0$; and
\item[(ii)] if $t=3$, then $\chi_d(\cL _{21})(0)={c\choose 2}{n+2\choose 2}-(c+2){c+1\choose 2}(n+1)+{c+2\choose 2}^2.$
\end{itemize}
\end{remark}

In the following table we illustrate  some values of $t,c,n-c\ge 2$ for which  $\Ext^1 _{\odi{X}}(\cL_2,\cL_1)\ne 0$.

$$
\begin{tabular}{c|c|c}
$t$ & $c$ & $d=n-c$ \\
\hline
2 & any & any \\
3 & 2 & $d \le $ 16 \\
3 & 3 & $d \le $ 10 \\
3 & $c \le 4$ & 8 \\
3 & $c \le 5$ & 7 \\
3 & $c \le 8$ & 6 \\
3 & $c \le 26$ & 5 \\
3 & any & $d\le 4$ \\
4 & $c\le 23$ & 3 \\
5 & $c\le 5$ & 3 \\
6 & $c\le 3$ & 3 \\
$t\le 9$ & 2 & 3 \\
4 & $c\le 5$ & 4 \\
\end{tabular}
\qquad \qquad
\begin{tabular}{c|c|c}
$t$ & $c$ & $d=n-c$ \\
\hline
5 & $c\le 3$ & 4 \\
6 & 2 & 4 \\
4 & $c\le 3$ & 5 \\
5 & 2 & 5 \\
4 & 2 & $d\le 8$  \\
$t\le 17$ & 2 & 2 \\
$t\le 11$ & 3 & 2 \\
$t\le 9$ & 4 & 2 \\
$t\le 8$ & 5 & 2 \\
$t\le 7$ & $c\le 8$ & 2 \\
$t\le 6$ & $c\le 26$ & 2 \\
$t\le 5$ & any & 2 \\
\end{tabular}
$$

\vskip 2mm
The above table shows a list of triples $(t,c,d)$ for which  the inequality
$\dim \Ext^{1}_{\odi{X}}(\cL_2,\cL_1)> 0$ holds. We can also check that $\dim
\Ext^{1}_{\odi{X}}(\cL_2,\cL_1) > 2$ unless $(t,c,d)\in \{ (2,2,2),(2,3,2),(3,2,16) \} $,
 since the polynomial $\chi_d(\cL _{21})(0)$ of Theorem
\ref{XdL21} takes the values $$\begin{cases} -1 & \text{ for } (t,c,d)=(2,2,2) \\-2 & \text{ for } (t,c,d)=(2,3,2)\\-2 & \text{ for } (t,c,d)=(3,2,16).
\end{cases}$$

\vskip 2mm
In the next Proposition we will deal with the case of curves and we will assume that $t>2$. This assumption is not at all restrictive since the case $t=1$ corresponds to $X=\PP^1$ and the case $t=2$ corresponds to a rational normal curve in $\PP^n$ both are varieties of finite representation type.

\begin{proposition}\label{casecurves} Let $X\subset \PP^n$, $n\ge 3$, be a general linear determinantal curve defined by the maximal minors of a $t\times (t+c-1)$ matrix $\cA$.  If  $t\ge 3$  then $$\dim \Ext^{1}_{\odi{X}}(\cL_2,\cL_1)\ge \frac{-1}{6}t (5 + 3t - 2 t^2)\ge 2.$$
\end{proposition}
\begin{proof} Since a general linear determinantal curve is smooth, $\cL_1$ and $\cL_2$ are line bundles and we have $\Ext^{1}_{\odi{X}}(\cL_2,\cL_1)=\Hl^1(X,\cL_1\otimes \cL_2^{\vee})$ and $\chi(\cL_1\otimes \cL_2^{\vee})=\h^0(X,\cL_1\otimes \cL_2^{\vee})-\h^1(X,\cL_1\otimes \cL_2^{\vee})$. Therefore, it is enough to check that $\chi(\cL_1\otimes \cL_2^{\vee})\le \frac{1}{6}t (5 + 3t - 2 t^2) \le -2.$
To this end, we will apply Riemann-Roch theorem. Since $p_a(X)=\sum  _{i=1}^{t-1}(i-1){n+i-2\choose i}$ (cf. \cite{Ha}), $YH^{n-2}={t+n-1\choose n}$ and $H^{n-1}={t+n-2\choose n-1}$ (cf. (\ref{degree})),  we have
$$\begin{array}{ccl} \chi(\cL_1\otimes \cL_2^{\vee}) & = & \chi (\odi{X}(-nY+(n+t-1)H)\\ & = & \de (\odi{X}(-nY+(n+t-1)H) +1-p_a(X) \\ & = & -n{t+n-1\choose n}+(n+t-1){t+n-2\choose n-1} +1-\sum  _{i=1}^{t-1}(i-1){n+i-2\choose i} \\ & =  &((-1 - n (-2 + t) + t) (n + t-2)!/(n!{\cdot} (t-1)!)
  \\& \le  & \frac{1}{6}t (5 + 3t - 2 t^2) \end{array} $$ which proves what we want.
\end{proof}

Finally we consider the case $c=1$ and $\dim X\ge 2$. Since $L_2=S_cM=M$ and $L_1=M^{\vee}(t-1)$, we have
$$\Ext^{i}_{\odi{X}}(\cL_2,\cL_1)\cong \ _{0} \! \Ext^{1+i}_R(S_{2}M(c),R(-1)) \text{ for } i=0,1$$
by Lemma \ref{ExtSiMSjMExtS2i-j-1+cMR}  because $\ell =\sum_{j=1}^ta_j-\sum_{i=1}^tb_i=t$. This implies

\begin{proposition} \label{$c=1$}
Let $X\subset \PP^n$ be a general linear determinantal hypersurface ($c=1$) defined by the determinant of a $t\times t$ matrix. Assume $n-1\ge 2$. Then $\Hom(\cL_2,\cL_1)=\Hom(\cL_1,\cL_2)=0$ and $\dim \Ext^{1}_{\odi{X}}(\cL_2,\cL_1)={t\choose 2}(n+1)-t^2.$ In particular if $t>2$, then $\dim \Ext^{1}_{\odi{X}}(\cL_2,\cL_1)\ge 3$, and if $t=2$, then $\dim \Ext^{1}_{\odi{X}}(\cL_2,\cL_1)=n-3$.
\end{proposition}
\begin{proof} Using Lemma \ref{HomL1L2}, we see that the $\Hom$-groups vanish. Now we argue as in the proof of Proposition
\ref{HomM-Mv} and we use the resolution
$$0\longrightarrow \wedge ^2F\longrightarrow F\otimes G\longrightarrow S_2G \longrightarrow S_2M \longrightarrow 0$$
and $\Hom(\cL_2,\cL_1)\cong \ _{0} \! \Ext^{1}_R(S_{2}M,R(-1))=0$ to get the exact sequence
$$0\longrightarrow (F^*\otimes G^*)(-1)_0\longrightarrow (\wedge^2F^*(-1))_0\longrightarrow \ _{0} \! \Ext^{2}_R(S_{2}M,R(-1))\longrightarrow 0$$
because $(S_2G)(-1)_0=R(-1)_0^{\beta }=0$ for some $\beta >0$. The exact sequence yields the following exact sequence
$$0\longrightarrow R_0^{t^2}\longrightarrow  R_1^{{t\choose 2}}\longrightarrow  \  _{0} \! \Ext^{2}_R(S_{2}M,R(-1))\longrightarrow 0,$$
cf. (\ref{cd}) and recall $\wedge ^{c+2}F=0$. Therefore,  we easily get the conclusions of the proposition.
\end{proof}


\section{Linear determinantal varieties of Ulrich wild representation type}

The goal of this section is to prove that under some numerical assumptions on
$n$, $c$ and $t$, general linear determinantal varieties will be examples of
varieties of arbitrary dimension for which wild representation type is
witnessed by means of Ulrich sheaves. In other words, they will have Ulrich
wild representation type.

\begin{construction}\label{construction} \rm With $\cL_i$ as in
  Notation~\ref{nota} we set $s_{21}:=\dim \Ext^1_{\odi{X}}(\cL_2,\cL_1)$. We
  assume $\dim X\ge 1$ and $s_{21}>0$, and we take $r-1\le s_{21}$
  $K$-linearly independent extensions $e_1, \cdots ,e_{r-1}\in
  \Ext^1_{\odi{X}}(\cL_2,\cL_1)$. We construct a rank $r$ sheaf $\shE$ on $X$
  sitting into the exact sequence:
\begin{equation}\label{a1}
  0\longrightarrow \cL_1 \longrightarrow \shE \longrightarrow \cL_2^{r-1} \longrightarrow 0 \end{equation} where $e_i$ is given by pullback of $ \shE \rightarrow \cL_2^{r-1}$ via the natural map $ \cL_2 \rightarrow \cL_2^{r-1}$ that takes an element onto its $i$-th coordinate. By construction $\shE$
is a rank $r$ initialized
($\Hl ^0 (X,\shE (-H))=0$) sheaf. Let us check that $\shE $
is an Ulrich sheaf. Since $\cL_1$ and $\cL_2$ are ACM, $\shE $ is also ACM
and $\shE $ is Ulrich because $\h^0 (\shE )=\h^0 (\cL_1)+\h^0(\cL_2)=
(r-1)\de (X)+\de (X)=\rk \shE {\cdot}\de X$. Moreover as in Construction \ref{construction1} one may check that   $\shE $ is indecomposable.
\end{construction}

\vskip 2mm Let $\shF _r$ be the irreducible family of rank $r$ indecomposable Ulrich  sheaves on $X$ given by (\ref{a1}). By construction $\shF _r\cong Gr(r-1, \Ext^1(\cL_2,\cL_1))$ being $Gr(r,V)$ the Grassmannian which parameterizes $r$-dimensional linear subspaces of V and $\dim \shF _r= (r-1)(s_{21}-r+1).$

Finally, it is worthwhile to point out that any $\shE\in \shF _r$ is $\mu$-semistable since any Ulrich sheaf is $\mu$-semistable but it is not $\mu$-stable because $\mu(\cL_1)=c_1(\cL_1)H=\frac{c_1(\shE)H}{2}=\mu(\shE)$.

\begin{proposition}\label{ext1E1E2bis} We keep the above notation and we consider $\shE _{i}$ the rank 2 Ulrich sheaves given by non-splitting extensions
\begin{equation}\label{extension}  e_{i}: \ \  0\longrightarrow \cL_1 \longrightarrow \shE _{i}\longrightarrow \cL_2\longrightarrow 0 \ \ \text{ for } \ i=1,2.\end{equation} Assume that $\dim X \ge 2$. Then, it holds:
$$\ext^1(\shE _1,\shE _2)\ge \ext^1(\cL_2,\cL_1)-2.$$
\end{proposition}

\begin{proof} It is analogous to the proof of Proposition \ref{ext1E1E2} since $\dim X\ge 2$ implies $\depth _{I(Z)}A\ge 3$ by Remark \ref{pamsrem}(i) and we omit it.
\end{proof}

\vskip 2mm
For the case of linear determinantal curves we have:

\begin{proposition}\label{ext1E1E2Curves} Let $X\subset \PP^n$, $n\ge 3$, be a general linear determinantal curve defined by the maximal minors of a $t\times (t+c-1)$ matrix $\cA$.
 We consider $\shE _{i}$ the rank 2 Ulrich bundles given by non-splitting extensions
\begin{equation}\label{extension1}  e_{i}: \ \  0\longrightarrow \cL_1 \longrightarrow \shE _{i}\longrightarrow \cL_2\longrightarrow 0 \ \ \text{ for } \ i=1,2.\end{equation} If  $t\ge 3$ then, it holds:
$$\ext^1(\shE _1,\shE _2)\ge  \ext^1(\cL_2,\cL_1)+p_a(X)-2\ge 3.$$
\end{proposition}
\begin{proof} Since a general linear determinantal curve $X\subset \PP^n$, $n\ge 3$, is smooth we have  $$\ext^1(\cL_1,\cL_1)\cong \h^1(X, \odi{X})=p_a(X)=\sum  _{i=1}^{t-1}(i-1){n+i-2\choose i}$$ where the last equality follows from \cite{Ha}; pg. 196.
We apply $\Hom(-,\cL_1)$ to the exact sequences (\ref{extension1}),  $i=1,2$, and we get the exact sequences:
$$\begin{array}{cccccccccc}0\rightarrow & \Hom(\cL_2,\cL_1) & \rightarrow & \Hom(\shE _{i},\cL_1) & \rightarrow &
\Hom(\cL_1,\cL_1)\cong K & \rightarrow & \Ext^1(\cL_2,\cL_1) & \rightarrow &
\\ & & & &  &
1 & \mapsto & e_{i} \end{array}$$
$$\rightarrow \Ext^1(\shE _{i},\cL_1)\rightarrow \Ext^1(\cL_1,\cL_1)\rightarrow 0.$$
Using them together with Lemma \ref{HomL1L2}, it follows that \begin{equation}\label{aux1} 0=\Hom(\cL_2,\cL_1)\cong \Hom(\shE _{i},\cL_1), \ \text{ for } \ i=1,2; \text{ and} \end{equation} \begin{equation}\label{aux2}\ext^1(\shE _{i},\cL_1) = \ext^1(\cL_2,\cL_1)+\ext^1(\cL_2,\cL_1)-1= \ext^1(\cL_2,\cL_1)+p_a(X)-1\ \text{ for } \ i=1,2.\end{equation}
Applying the functor $\Hom(-,\cL_2)$ to the exact sequence (\ref{extension1}) and using Lemma \ref{HomL1L2}, we get the exact sequence:
$$0\rightarrow \Hom(\cL_2,\cL_2)\cong K \rightarrow  \Hom(\shE _{1},\cL_2)  \rightarrow
\Hom(\cL_1,\cL_2)=0$$ which allows us to deduce that
\begin{equation}\label{aux}
\Hom(\shE _{1},\cL_2)\cong \Hom(\cL_2,\cL_2)\cong K. \end{equation}
Now, we apply the functor $\Hom(\shE _1,-)$ to the exact sequence (\ref{extension1})
and we get the exact sequence:
$$0\rightarrow \Hom(\shE _1,\cL_1)  \rightarrow  \Hom(\shE _{1},\shE _2)  \rightarrow
\Hom(\shE _1,\cL_2) \rightarrow  \Ext^1(\shE _1,\cL_1)  \rightarrow  \Ext^1(\shE _{1},\shE _2)\rightarrow \Ext^1(\shE _1,\cL_2)\rightarrow 0.$$
Thus, using (\ref{aux1})-(\ref{aux}), we obtain
$$\begin{array}{rcl} \ext^1(\shE _{1},\shE _2) & = & -\hom(\shE _1,\cL_1)+\hom(\shE _{1},\shE _2)-\hom(\shE _1,\cL_2)+\ext^1(\shE _1,\cL_1)+ \ext^1(\shE _1,\cL_2) \\
& \ge & -\hom(\shE _1,\cL_1)+\hom(\shE _{1},\shE _2)-\hom(\shE _1,\cL_2)+\ext^1(\shE _1,\cL_1)
\\& = & \ext^1(\cL _2,\cL_1)-2+p_a(X)\ge 3
\end{array}$$
where the last inequality follows from the fact that for $t\ge 3$, $\ext^1(\cL _2,\cL_1)\ge 2$ (cf.  Proposition \ref{casecurves}) and $p_a(X)=\sum  _{i=1}^{t-1}(i-1){n+i-2\choose i}\ge 3$.
\end{proof}

We are now ready to state a theorem which together with Theorem \ref{XdL21}, Remark \ref{ext1(L2,L1)t=2,3} and Proposition \ref{casecurves} lead to the main results of this section.

\begin{theorem} \label{mainthm11} Let $X\subset \PP^n$ be a general linear
  standard determinantal scheme. Assume that $\dim X\ge 1$. If $\dim
  \Ext^{1}_{\odi{X}}(\cL _2,\cL _1)> 2$ then $X$ is of Ulrich wild
  representation type.
\end{theorem}

\begin{remark} \rm  According to Definition 2 and Corollary 1 in \cite{FPL}  we have proved a slightly stronger result. In fact, we  have proved that  with the above assumptions $X$ is strictly Ulrich wild. We thank the referee for pointing out this improvement.
\end{remark}

\begin{proof} First of all we observe that the hypothesis $\dim
  \Ext^{1}_{\odi{X}}(\cL _2,\cL _1)> 2$ implies $t\ge 3$ provided $\dim X=1$. Therefore, we can
  use the family ${\mathcal F}_2$ above and $\dim \Ext^1(\shE
  _1,\shE _2)>0$, see Propositions \ref{ext1E1E2bis}
  and \ref{ext1E1E2Curves}, to construct iterated extensions leading to large
  families of simple Ulrich sheaves similarly to what we did in Theorem
  \ref{Bigthm}. The proof is analogous to the proof of Theorem \ref{Bigthm}
  and we omit it.
\end{proof}

Summarizing, we have got

\begin{corollary}\label{maincoro} Let $X\subset \PP^n$ be a general linear standard determinantal scheme of codimension $c$ defined by the maximal minors of a $t\times (t+c-1)$ matrix with linear entries. Assume that one of the following conditions holds:
\begin{itemize}
\item[(1)] $(n-c,t)=(2,2)$ and $c > 3$; or
\item[(2)] $t=2$, $n-c>2$ and $c \ge 2$; or
\item[(3)] $t=2$, $n>5$ and $c =1$, or
\item[(4)] $t>2$, $n>2$ and $c=1$.
\end{itemize}
Then $X$ is of Ulrich wild representation type.
\end{corollary}
\begin{proof} (1) and (2) follow from Theorem \ref{mainthm11} and Remark \ref{ext1(L2,L1)t=2,3} (i), cf. the table and text after Remark \ref{ext1(L2,L1)t=2,3} for $t=2$.

(3) and (4) follow from  Theorem \ref{mainthm11} and Proposition \ref{$c=1$}.
\end{proof}

\begin{remark}\rm The case $(n-c,t,c)=(2,2,2)$ corresponds to a cubic scroll
  $X$ in $\PP^4$ which has finite representation type. Also in the case
  $(n-c,t,c)=(2,2,3)$  Theorem
  \ref{mainthm11} does not apply since $\dim \Ext^{1}_{\odi{X}}(\cL_2,\cL_1) =
  2$; it corresponds to a quartic  scroll  in $\PP^5$ which has tame representation type (see \cite{FM}). Note that the case $n=3$ of (4) is also proved in \cite{Casn}.
\end{remark}

\begin{corollary}\label{maincoro3} Let $X\subset \PP^n$ be a general linear standard determinantal scheme of codimension $c\ge 1$ defined by the maximal minors of a $t\times (t+c-1)$ matrix with linear entries. Assume that $n-c\ge 2$, $t=3$ and $$-{c\choose 2}{n+2\choose 2}+(c+2){c+1\choose 2}(n+1)-{c+2\choose 2}^2>2.$$
Then $X$ is of Ulrich wild representation type. In particular, $X\subset \PP^n$ is of Ulrich wild representation type if $t=3$ and $n\le 12$.
\end{corollary}
\begin{proof} This follows from Theorem \ref{mainthm11}, Theorem  \ref{XdL21}, Proposition \ref{$c=1$} and Remark \ref{ext1(L2,L1)t=2,3} (ii). Moreover, any $(c,n-c)$ with $c\ge 2$, $n-c\ge 2$, $n\le 12$ and $t=3$ is quoted in the table after Remark \ref{ext1(L2,L1)t=2,3} and satisfies $\dim\Ext^{1}_{\odi{X}}(\cL _2,\cL _1)> 2$, and we conclude by Theorem  \ref{mainthm11}.
\end{proof}

\begin{corollary}\label{maincoro4} Let $X\subset \PP^n$ be a general linear standard determinantal scheme of codimension $c$ defined by the maximal minors of a $t\times (t+c-1)$ matrix with linear entries.
Then $X$ is of Ulrich wild representation type provided $(t,c,n-c)$ belong to the table after Remark \ref{ext1(L2,L1)t=2,3} and $(t,c,n-c)\notin \{ (2,2,2),(2,3,2),(3,2,16) \}$.
\end{corollary}
\begin{proof} This follows from Theorem   \ref{mainthm11}, and the table and the text after Remark \ref{ext1(L2,L1)t=2,3}.
\end{proof}

\begin{theorem}\label{mainthm2} All general non-degenerated linear determinantal curves $X\subset \PP^n$, $n\ge 3$,  are of Ulrich wild representation type except for  the rational normal curves which are of finite representation type.
\end{theorem}
\begin{proof}
For general linear determinantal curves $X \subset  \PP^n$ other that $\PP^1$ or a rational normal curve, we have $\dim \Ext^1(\cL_2,\cL_1)\ge 2$
(see Proposition \ref{casecurves}). Hence $X$ is of Ulrich wild representation type by Theorem \ref{mainthm11}.
\end{proof}


\section{Final comments and a Conjecture.}
 We end the paper with a Conjecture raised by this paper. In fact
Theorems \ref{mainthm11} and \ref{mainthm2} and Corollaries \ref{maincoro}, \ref{maincoro3} and \ref{maincoro4} together with  examples computed with Macaulay2 (\cite{Mac2})
suggest - and prove in
many cases - the following conjecture:

\begin{conj}\label{conjecture} A general  linear standard determinantal scheme $X\subset \PP^n$ of codimension $c\ge 2$ defined by the maximal minors of a $t\times (t+c-1)$ matrix is of
Ulrich wild representation type unless $X$ is $\PP^{n-c}$, the rational normal curve  in $\PP^n$ or the cubic scroll  in $\PP^4$ which are of finite representation type; or the quartic scroll in $\PP^5$ which is of tame representation type.
\end{conj}

\begin{remark} \rm
(i) Using Macaulay2 we can enlarge the list of integers $(t,c,d=n-c)$ such that $\dim \Ext^1(\cL_2,\cL_1 )>2$ and conclude (using our approach of building  families of Ulrich bundles of arbitrary high rank and dimension) that, in addition to the cases covered by Theorems \ref{mainthm11} and \ref{mainthm2} and Corollaries \ref{maincoro}, \ref{maincoro3} and \ref{maincoro4}, there are many other cases which support Conjecture 6.1. Indeed, for $(t=3,c=2,18\le n\le 40)$ we have got $\dim \Ext^1(\cL_2,\cL_1 )>2$.

(ii) One of the first open cases corresponds to $(t,c,d)=(5,3,5)$, i.e. a linear standard determinantal variety $X$ of dimension 5  in $\PP^8$ defined by the maximal minors of a $5\times 7$ matrix with linear entries.

(iii)  It is worthwhile to point out that we do not claim that for any $(t,c,d=n-c)$ we have $\dim \Ext^1( \cL_2,\cL_1)>2$. In fact, using Macaulay2, we have checked that for $(t,c,n)=(3,3,14\le n \le 21) \text{ or } (3,4, 17\le n \le 21) \text{ or } (4,2,19\le n \le 21)$ we have $\dim \Ext^1(\cL_2,\cL_1)=0$. In these cases, to prove Conjecture 6.1 we cannot proceed as in Theorem \ref{mainthm11} and another approach is required.

\end{remark}

As another interesting open problem we propose the construction of Ulrich sheaves on standard determinantal varieties $X\subset \PP^n$ of codimension $c$ defined by the maximal minors of a $t\times (t+c-1)$ matrix with entries homogeneous forms of degree $d_{ij}\ge 0$. To our knowledge   there are no known examples of Ulrich sheaves on $X$ unless $X$ is either a complete intersection or a linear standard determinantal scheme or a codimension 2 standard determinantal scheme.


\end{document}